\documentclass[11pt]{article}

\usepackage{amsmath, amsfonts, amssymb, amsthm}
\usepackage{psfrag,epsfig}
\usepackage{color}
\usepackage[margin=.94in]{geometry}
\usepackage{filecontents}

\newcommand{\bsigma}{\boldsymbol{\sigma}}

\newcommand{\bQ}{\boldsymbol{Q}}

\newcommand{\bW}{\boldsymbol{W}}

\newcommand{\bA}{\boldsymbol{A}}

\newcommand{\bx}{\mathsf{x}}

\newcommand{\oracle}{\mathsf{oracle}}

\newcommand{\ba}{\boldsymbol{a}}

\newcommand{\by}{\boldsymbol{y}}

\newcommand{\bm}{\boldsymbol{m}}
\newcommand{\blambda}{\boldsymbol{\lambda}}

\newcommand{\brho}{\boldsymbol{\rho}}

\newcommand{\bX}{\boldsymbol{X}}

\newcommand{\argmax}{\operatornamewithlimits{arg\,max}}

\newcommand{\E}{\mathbb E}
\newcommand{\pr}{\mathbb {P}}

\newcommand{\var}{\operatorname{Var}}

\newcommand{\tonghoonnote}[1]{{\color{red}{#1}}}

\newtheorem{theorem}{Theorem}
\newtheorem{lemma}[theorem]{Lemma}
\newtheorem{corollary}[theorem]{Corollary}
\newtheorem{proposition}[theorem]{Proposition}

\theoremstyle{definition}

\newtheorem{definition}{Definition}

\begin{document}

\title{Scheduling using Interactive  Optimization
  Oracles\\ for Constrained Queueing Networks\thanks{A preliminary version of this work is presented as a poster at ACM International Conference 
on Measurement and Modeling of Computer Systems (SIGMETRICS), 2014.}}

\author{
 Tonghoon Suk\thanks{T.\ Suk is with 
H. Milton Stewart School of Industrial \& Systems Engineering,
   Georgia Institute of Technology,
   Atlanta, GA 30332, USA.
   Email: tonghoons@gatech.edu.}
\and
 Jinwoo Shin\thanks{
J. Shin is with School of Electrical Engineering,
   Korea Advanced Institute of Science and Technology,
   Republic of Korea.
   Email: jinwoos@kaist.ac.kr.}
}

\date{}

\maketitle

\begin{abstract}
Ever since Tassiulas and Ephremides (1992) proposed the maximum weight scheduling algorithm of throughput-optimality for constrained queueing networks that arise in the context of communication networks, extensive efforts have been devoted to resolving its most important drawback: high complexity. This paper proposes a generic framework for designing throughput-optimal and low-complexity scheduling algorithms for constrained queueing networks. Under our framework, a scheduling algorithm updates current schedules by interacting with a given oracle system that generates an approximate solution to a related optimization task. One can utilize our framework to design a variety of scheduling algorithms by choosing an oracle system such as random search, Markov chain, belief propagation, and primal-dual methods. The complexity of the resulting scheduling algorithm is determined by the number of operations required for an oracle to process a single query, which is typically small.
We provide sufficient conditions for throughput-optimality of the scheduling algorithm in general constrained queueing network models.
The linear-time algorithm of Tassiulas (1998) and the random access algorithm of Shah and Shin (2012)
correspond to special cases of our framework using random search and Markov chain oracles, respectively.
Our generic framework, however, provides a unified proof with milder assumptions.
\end{abstract}

\section{Introduction}
The dynamic resource allocation problem in modern communication networks such as wireless networks and input queued switches, examples of constrained queueing networks in which only certain sets of queues can be served simultaneously, is often addressed by the maximum-weight scheduling (MWS) algorithm. As it is throughput-optimal, MWS algorithm yields a stable system under all possible loads for which it can be made stable and requires information only about current queue lengths. However, because it requires repeatedly solving computationally hard problems to find ``good'' schedules, the MWS algorithm cannot be implemented in practice.
Therefore, extensive research has proposed throughput-optimal scheduling algorithms with low complexity.
Examples of such algorithms include simpler implementations of the MWS algorithm \cite{T98, GPS03, MSZ06}, greedy algorithms \cite{AOST93, M99, JLS08, LNS09}, and random access algorithms \cite{GM96, GS06, MEO07}.

\subsection{Our Contribution} 
This paper introduces a novel framework for designing low-complexity throughput-optimal scheduling algorithms in constrained queueing networks,
by utilizing iterative optimization methods 
approximating a ``good'' schedule (i.e., a maximum-weight schedule). 
While the standard implementation of the MWS algorithm entails all iterations of such a method at each service time, the scheduling algorithm in our framework entails only one iteration of it at each service time, which means that the computational time required to find a schedule decreases significantly. Furthermore, we show that the scheduling algorithm preserves throughput-optimality. To build our generic framework, 
we view steps of an iterative optimization methods as queries to a black box that we formalize as an interactive oracle system.  The input of the oracle system depends on the current state of network system, and the output consists of a schedule and ``advice'', information used in the next step of the method. We describe four examples of the oracle system: random search (RS), Markov chain Monte Carlo (MCMC), belief propagation (BP), and primal-dual methods (PDM). For instance, for MCMC, the advice given by the oracle consists of the state of the Markov chain and the current schedule.
After formulating an oracle system from any iterative optimization method, 
one can design a throughput-optimal and low-complexity scheduling algorithm 
via interacting with it.

The intuitive reason why one step of an approximation method suffices for throughput-optimality follows. This method seeks a schedule of maximum weight, which is a function of the queue lengths. We construct a weight function such that its value remains constant for long stretches of time.  Therefore, although we only use one step of the method at each service time, the schedule automatically approximates a maximum weight schedule as time passes, which guarantees the throughput-optimality of the algorithm. 
This underlying intuition is similar in spirit to that in \cite{RSS09, SS12}. The main difference is that 
while the authors in \cite{RSS09, SS12} 
force the weight function value ``vary slowly'' in real numbers, we let them ``vary rarely'' in integers. Because we introduce an integer-valued weight function,  we do not need to analyze ``time-varying'' systems, which simplifies the throughput-optimality proof. 
More importantly, our proof is robust in the sense that it is not sensitive to the given oracle systems, underlying network structures, and arrival processes, as explained in Section~\ref{sec:main}. 

Our generic framework overcomes several limitations of previous work.
First, most existing throughput-optimal algorithms \cite{GPS03,MSZ06,RSS09,SS12}
rely on an underlying network structure, and in principle, they are not easily applied to networks with other structures. In addition, 
proving their throughput-optimality requires a unique set of techniques for each algorithm. 
In contrast, our generic framework does not rely on a network structure, and it guarantees 
throughput-optimality 
by only checking simple algebraic conditions.
Furthermore, the authors of \cite{RSS09, SS12}  considered only Bernoulli arrival processes, and their proofs are not easily generalizable to other arrival processes. However, the  algorithm resulting from our framework is throughput-optimal under any arrival processes with bounded second moments.


One way in which our framework can be used is to select a low-complexity, throughput-optimal scheduling algorithm with good delay performance. 
Using our framework, one can establish the throughput optimality of a family of scheduling algorithms 
that interact with optimization methods and measure 
their delay performance through simulation. Therefore, one can test which algorithm works best in practice 
while theoretically guaranteeing throughput-optimality. 

\subsection{Related Work} 
Simpler or distributed implementations of the MWS algorithm have been extensively proposed in the literature. 
Tassiulas \cite{T98} provides the so-called ``pick-and-compare'' algorithm, which is a linear-complexity version of the MWS algorithm but suffers from bad delay performance.
The work in this line also includes
a variant of the MWS algorithm by Giaccone, Prabhakar, and Shah \cite{GPS03} and 
a gossip-based algorithm by Modiano, Shah, and Zussman \cite{MSZ06}. However, these algorithms are specific to certain network models
and still 
require numerous information (or message) exchanges 
for each new scheduling decision. 
Recently, 
even fully distributed random access
algorithms have been shown to achieve desired high performance (i.e., throughput-optimality)
in both wireless interference and buffered circuit switched network models \cite{RSS09, SS12}.
The main intuition underlying these results is that nodes in a network can adjust
their random access parameters dynamically using local information
such as queue lengths so that they can simulate the MWS algorithm asymptotically 
for throughput-optimality. From an optimization point of view, under these algorithms, nodes run a Markov chain Monte Carlo (MCMC) with time-varying
parameters depending on queue lengths. If the parameters change slowly enough, the authors of \cite{SS12} proved that algorithms sample a maximum-weight schedule (for throughput-optimality).
We note that the ``pick-and-compare'' algorithm and
the random access algorithm 
can also be understood
as special cases of algorithms developed under our generic framework using RS and MCMC oracles, respectively, and more details appear in Section \ref{sec:exam}.

Although several greedy algorithms reduce time complexity, they achieve only some fraction of the maximal throughput region.  For example, parallel iterative matching \cite{AOST93} and iSLIP \cite{M99} have been shown to be 50\% throughput optimal \cite{DP00}.  In addition, Kumar et al. \cite{KGL02} and Dimakis and Walrand \cite{DW06} identified sufficient conditions on the network topology for throughput-optimality. Joo et al. \cite{JLS08} and Leconte et al. \cite{LNS09} further analyzed these conditions to obtain fractional throughput results for a class of wireless networks. However, these algorithms are generally not throughput optimal and require multiple rounds of message exchanges between nodes.

\subsection{Organization}
Section \ref{sec:pre} describes the constrained queueing network model of  interest in this study
and the performance metric (i.e., throughput-optimality) for scheduling algorithms.
Section \ref{sec:main} provides the main results of this paper: a generic framework for
designing a throughput-optimal and low-complexity scheduling algorithm that finds its current schedule via interaction with an oracle system. It also states the throughput-optimality proof with an associated key lemma.
Section \ref{sec:exam} introduces several examples of scheduling algorithms under our framework, and Section \ref{sec:pflemmas} presents the formal proof of the key lemma.

\section{Preliminaries}\label{sec:pre}

\subsection{Network Model}\label{subsec:network_model}

The constrained queueing network, a stochastic network system with service-level constraints, consists of many buffers that temporarily store packets (jobs) to be served. Packets arrive at each buffer via an exogenous stochastic process and leave the system after being served. At most one packet in each nonempty buffer can be served at a time, and all packets have a unit service time. However, because of service constraints, not all nonempty buffers can transmit their packets simultaneously, and only certain subsets of the buffers can serve packets at the same time. We call these subsets \emph{schedules}, and every constrained queueing network has its own collection of schedules. At each service epoch, any scheduling algorithm selects a schedule among the collection, and nonempty buffers in the schedule process their packets. Our goal is to design scheduling algorithms that require little computational time to choose a schedule at each service epoch while maintaining high performance. Our performance metric introduced in the next section relates to the number of packets (\emph{queue length}) in each buffer. For the next step, we set up a mathematical model that represents the above network system and describe how the queue length of each buffer changes as time evolves.

Our model is a constrained queueing network with $n$ buffers in time slotted by service epochs (i.e., 
time is denoted by a nonnegative integer variable $t\in\mathbb Z_+:=\{0,1,\dots,\}$), and at each time $t$, a schedule is selected by a scheduling algorithm.
Buffers are indexed by elements in the set $\mathcal{I}$ ($|\mathcal{I}|=n$), and the queue length of buffer $i\in\mathcal{I}$ is denoted by $Q_i(t)$. 
Now, we show how $Q_i(t+1)$ changes from $Q_i(t)$ by arrivals and service. 
During time interval $[t,t+1)$, the queue length of buffer $i$ increases by the number of (external) arrival packets 
at buffer $i$ and decreases by $1$ if a selected schedule (a subset of buffers) at time $t$  contains buffer $i$. For a mathematical illustration of this observation, we denote the number of arrivals to buffer $i$ during $[t,t+1)$ by $A_i(t)$ and
depict a schedule by an $n$-dimensional binary vector $\bsigma=[\sigma_i:i\in \mathcal{I}]$ such that $\sigma_i=1$ if buffer $i$ is in the schedule, and $\sigma_i=0$ otherwise.  We also let $\mathcal{S}\subset \{0,1\}^n$ be the set of all available schedules and $\bsigma(t)\in \mathcal{S}$ the 
schedule during $[t,t+1)$ for $t\in\mathbb Z_+$. Then, the above observation is expressed as 
\begin{equation}\label{eq:dynamics_of_Q}
Q_i(t+1)~=~Q_i(t) + A_i(t) -\sigma_i(t)\,\mathbb{I}_{\{Q_i(t)>0\}},
\end{equation}
where $\mathbb{I}_\mathcal{A}$ is an indicator function of event $\mathcal{A}$. We close this section with key assumptions relating to the external arrivals of packets: $\big\{A_i(t)\in\mathbb Z_+:t\in \mathbb Z_+, \,i\in \mathcal{I}\big\}$ are independent random variables with
$$ \E[A_i(t)]=\lambda_i,\quad \var[A_i(t)]\leq \mu^2,$$
where $\lambda_i\in [0,1]$ is the \emph{arrival rate} for buffer $i$, and $\mu>0$ is {a} positive (finite) constant.

\subsection{Performance Metric}

Our goal is to design high-performance scheduling algorithms that find a schedule $\bsigma(t)\in \mathcal{S}$ at each time $t\in\mathbb Z_+$ in little computational time.
In this paper, a scheduling algorithm has high performance, called \emph{throughput-optimality} if 
it ensures that queues do not blow up as long as the vector of arrival rates is within the system maximal stability region.

To describe it formally, we define the capacity region as follows:
$$\mathcal{C}\,:=\,\left\{\sum_{\bsigma\in S}\alpha_{\bsigma}\,\bsigma\,:\, 
\sum_{\bsigma\in \mathcal{S}}\alpha_{\bsigma} =1\mbox{ and }\alpha_{\bsigma}\geq 0\mbox{ for all } \bsigma\in \mathcal{S}\right\},$$
that is, the convex hull of all available schedules in $\mathcal S$.
The capacity region $\mathcal{C}$ essentially contains all effective service rates induced by any scheduling algorithm.
Therefore, if queues in a system with arrival rate vector $\blambda$ are stable by any scheduling algorithm, 
there exists $\bsigma\in \mathcal{C}$ such that $\blambda\leq \bsigma$ component-wise; we call such $\blambda$ \emph{admissible}. Also, when arrival rate vector $\blambda$ is strictly less than some $\bsigma$ in $\mathcal{C}$, we say $\blambda$ is \emph{strictly} admissible, and the set of all strictly admissible arrival rate vectors is denoted by $\Lambda^o$:
\begin{equation*}
  \Lambda^o~:=~\big\{ \blambda\in\mathbb R_+^n\,:\,\blambda<\bsigma,\mbox{ for some }\bsigma\in\mathcal{C} \big\}.
\end{equation*}
Thus, a throughput-optimal scheduling algorithm is able to make a system \emph{stable} for any arrival rates $\blambda \in \Lambda^o$, which is formally stated as follows.

\begin{definition}\label{def:throughput-optimality}
	A system is \emph{stable} if  
	$$\liminf_{t\to\infty} \sum_{i\in\mathcal{I}}Q_i(t)<\infty\qquad\textrm{with probability $1$},$$i.e., the total queue length remains finite with probability $1$.
	A scheduling algorithm is called {\em throughput-optimal} if
	the system with arrival rates vector $\blambda \in \Lambda^o$ is stable under the scheduling algorithm.
\end{definition}
\noindent
To prove that scheduling algorithms from our framework
are throughput optimal, we first define an appropriate
underlying Markov chain and show that a subset of states with bounded total queue length is positive recurrent utilizing the popular Lyapunov-Foster criteria, which is introduced in the following section.

\subsection{Stability of a System: Lyapunov-Foster Criterion}\label{sec:lyapunov}

This section introduces a method for proving the positive recurrence in a Markov chain and its relation to the stability of a system,
which is proved by the conclusion in Lemma~\ref{thm:Lyapunov}.
We first recall the definition of the positive recurrence in Markov chain $\{\bX(t)\,:\,t\in\mathbb Z_+\}$ on state space $\Omega$. 
A subset $\mathcal{B}\subset \Omega$ is said to be {\em recurrent} if $\inf_{\bx\in \mathcal{B}} \pr \left[ \tau_B<\infty\,|\,\bX(0)=\bx\right]=1$, where $\tau_\mathcal{B}=\inf\{t\geq 1\,|\,\bX(t)\in\mathcal{B}\}$ is a hitting time for $\mathcal{B}$. If $\sup_{\bx\in \mathcal{B}}\E\left[ \tau_\mathcal{B}\,|\,\bX(0)=\bx\right]<\infty$, recurrent subset $\mathcal{B}$ is called {\em positive recurrent}. 
One way to show the positive recurrence is to use the following negative drift condition on a Lyapunov function, also known as the Lyapunov-Foster criterion.

\begin{lemma}[\protect{\cite[Theorem 1]{FK04}}]\label{thm:Lyapunov}
Let $\{\bX(t)\,:\,t\in\mathbb Z_+\}$ be a Markov chain on state space $\Omega$, and
$L:\Omega\to\mathbb R_+$ be a function on $\Omega$ such that
$\sup_{\bx\in\Omega} L(\bx)=\infty$.
For any $\gamma\geq 0$, define
$\mathcal{B}_\gamma=\{\bx\in\Omega \,:\,L(\bx)\leq\gamma\}$.
  Suppose there exist functions $\tau,\kappa:\Omega\to\mathbb R_+$ such that 
  \begin{equation}\label{eq:thm:Lyapunov}
    \E[L(\bX(\tau(\bx)))-L(\bX(0))\,|\,\bX(0)=\bx]\leq -\kappa(\bx),~\forall \bx\in\Omega,c
  \end{equation}
and they satisfy the following conditions:
  \begin{enumerate}
    \item[{\bf L1.}] $\liminf_{L(\bx)\to\infty} \kappa(\bx)>0$.
    \item[{\bf L2.}] $\inf_{\bx\in \Omega} \kappa(\bx)>-\infty$.
    \item[{\bf L3.}] $\sup_{\bx\in B_{\gamma}} \tau(\bx) <\infty$ for all $\gamma\in\mathbb R_+$.
    \item[{\bf L4.}] $\limsup_{L(\bx)\to\infty} \tau(\bx)/\kappa(\bx)<\infty$. 
  \end{enumerate}
  Then, there exists constant $\gamma_0 > 0$ so that for all
$\gamma_0 < \gamma$, the following holds:
\begin{eqnarray*}
\sup_{\bx \in \mathcal B_\gamma} \E\left[ T_{B_\gamma} \,|\,\bX(0)=\bx\right] & < & \infty. 
\end{eqnarray*}
Namely, $\mathcal{B}_\gamma$ is positive recurrent. 
\end{lemma}
\noindent The above function $L$ is called a \emph{Lyapunov function}.
To show a system is stable, we construct an underlying network Markov chain and define a Lyapunov function that depends on queue lengths and  goes to infinity as total queue length goes to infinity. If $\mathcal{B}_\gamma$ is positive recurrent for any $\gamma>\gamma_0$, then the system is stable by the following argument:
Let the initial state be $\bx\in B_\gamma$ for some $\gamma>\gamma_0$. 
Since $B_\gamma$ is positive recurrent, the Markov chain hits $B_\gamma$ infinitely often with probability $1$, which implies that the system is stable because the total queue length of any state in $B_\gamma$ is bounded. Therefore, to guarantee throughput-optimality of our scheduling algorithm, for any arrival rate vector $\blambda \in \Lambda^o$, we need to find $\tau$ and $\kappa$, which satisfies \eqref{eq:thm:Lyapunov} and conditions {\bf L1}--{\bf L4}.


\section{Scheduling using Interactive Oracles}\label{sec:main}

This section presents our main results, a general framework for designing low-complexity scheduling algorithms for constrained queueing networks and the sufficient conditions for throughput-optimality of the algorithms. As introduced in Section~\ref{subsec:network_model}, a constrained queueing network is represented by $(\mathcal{I},\mathcal{S})$: $\mathcal{I}$ is an index set for buffers ($|\mathcal{I}|=n$), and $\mathcal{S}$ is the set of all schedules that are $n$ dimensional binary vectors. For such system, a well-known throughput-optimal scheduling algorithm is the maximum-weight scheduling (MWS) algorithm \cite{TE92}, which selects a solution (schedule) to the following optimization problem:
\begin{equation} 
  \max \left\{ \brho\cdot\bW:=\sum_{i\in \mathcal{I}} \rho_i W_i\,:\,\brho\in \mathcal{S}\right\},\label{eq:max_problem}
\end{equation}
where $\bW$ is an $n$-dimensional vector called a \emph{weight vector} and $\brho\cdot\bW$ is called the \emph{weight of schedule $\brho$}. {Namely, the optimization problem \eqref{eq:max_problem} finds a maximum-weight schedule in $\mathcal{S}$.}
Weight vector $\bW$ depends on queue length vector $\bQ(t):=[Q_i(t)\,:\,i\in\mathcal{I}]$ and at every service epoch, the MWS algorithm requires solving the above optimization problem. For weight vector $\bW$, a solution to optimization problem~\eqref{eq:max_problem} can be obtained by various methods according to the structure of network system $(\mathcal{I},\mathcal{S})$. Such a method usually consists of many steps (iterations) that induce a long computation time at each service epoch in the MWS algorithm. As the MWS algorithm, the scheduling algorithm in our framework utilizes an iterative method for optimization problem~\eqref{eq:max_problem}, but uses only one step per a service epoch instead of all steps in the method. Thus, the algorithm takes little computational time to find a schedule {at each service epoch}. In addition, proper choices of weight vector $\bW$ at each service epoch guarantees the throughput-optimality of the algorithm.
 In the remainder of this section, we describe the algorithm in detail: Section~\ref{sec:oracle_system} introduces a general (abstract) concept of one step (iteration) of the method that solves problem~\eqref{eq:max_problem}, Section~\ref{sec:algorithm_description} describes our scheduling algorithm and conditions that guarantee the throughput-optimality of the algorithm, and Section~\ref{sec:pfmainthm} presents 
 the proof outline of our main theorem.

\subsection{Oracle System}\label{sec:oracle_system}
To develop throughput-optimal, low-complexity scheduling algorithms for a constrained queueing network represented by $(\mathcal{I},\mathcal{S})$, this paper proposes an algorithm that finds a schedule in $S$ at each service epoch by utilizing a black box called an \emph{oracle system}. The oracle system is motivated by one iteration in (randomized or deterministic) iterative methods for finding an (approximate) optimal solution to optimization problem~\eqref{eq:max_problem}. Typically, at every iteration, an iterative method updates its current solution (schedule) using information from the previous iteration (and weight vector $\bW$); we refer to such information transmitted between two consecutive iterations \emph{advice}. Thus, an iterative method can be understood as a process interacting with a black box that receives advice as an input and outputs an updated schedule and new advice used in the next iteration; that is, the iterative method maintains advice (and a weight vector), and at each iteration, it sends current advice to the black box and replaces the current advice and the current schedule with outputs from the black box. We introduce a generalized definition of the black box in an iterative method, the oracle system, which has the following input and output:
\begin{itemize}
	\item[$\circ$] The oracle system receives advice $\ba$ and weight vector $\bW=[W_i\in \mathbb Z_+:i\in \mathcal{I}]$ as inputs,
	\item[$\circ$] The oracle system outputs (or returns) schedule $\bsigma\in\mathcal{S}$ and updated advice $\widehat{\ba}$.
\end{itemize}
We denote the set of all advice by $\mathcal {A}$. 
Since the oracle system is similar to one step (iteration) in an iterative method that finds an (approximate) solution to optimization problem~\eqref{eq:max_problem}, when we consecutively interact with the oracle system while fixing a weight vector, we obtain an approximate solution. To state this argument formally, when the oracle system takes advice $\ba$ and weight vector $\bW$ as inputs, 
we denote outputs by $\bsigma =\bsigma_{\oracle}(\ba)= \bsigma_{\oracle}(\bW,\ba)$ and $\widehat{\ba}=\ba_{\oracle}(\ba)=\ba_{\oracle}(\bW,\ba)$, where the oracle can generate random outputs in general. 
Then, we assume that the oracle system satisfies the following condition:
\begin{enumerate}
  \item[{\bf C0.}] For any $\eta,\delta\in(0,1)$, if $W_{\max}:=\max_{i\in \mathcal{I}} W_i$ is large enough,
there exists $h=h(W_{\max},\eta,\delta)$ such that for any $t\geq h$ and advice $\ba\in\mathcal A$,
  	\begin{align*}
\left( \bsigma_{\oracle}(\ba_{\oracle}^{(t)}(\ba))\right)
  	\cdot \bW
  	 &~\geq~(1-\eta) \max_{\brho\in S}
  	\brho\cdot \bW
\qquad\mbox{with probability at least $ 1-\delta$},
  	\end{align*}
  	where $\ba_{\oracle}^{(t)}$ is the function composing $\ba_{\oracle}$ ``$t$ times'' (i.e., $\ba_{\oracle}^{(t)} = \ba_{\oracle}^{(t-1)}\circ \ba_{\oracle}$).
\end{enumerate}
Condition {\bf C0} implies that  
{after $h$ interactions, the oracle system generates schedule $\bsigma$ that is an approximate solution to \eqref{eq:max_problem}.}

\subsection{Scheduling Algorithm}\label{sec:algorithm_description}
This section describes how our scheduling algorithm interacts with an oracle system that corresponds to one step (iteration) in an iterative method for  optimization problem~\eqref{eq:max_problem}. The oracle system receives advice and a weight vector as inputs. Our scheduling algorithm maintains advice $\ba(t)$ and weight vector $\bW(t)$
along with queue length vector $\bQ(t)$. 
Then, at service epoch $t$, current advice $\ba(t)$ and current weight vector $\bW(t)$ are sent to the oracle system, which returns updated advice $\ba(t+1)$ and schedule $\bsigma(t)$. 
Then, schedule $\bsigma(t)$ {and}  arrival vector $\bA(t):=[A_i(t)\,:\,i\in\mathcal{I}]$ during $[t,t+1)$ determine queue length vector $\bQ(t+1)$ at time $t+1$ by \eqref{eq:dynamics_of_Q}. 
Therefore, the time-complexity of the scheduling algorithm is precisely depends on
how long the oracle system takes to process a query (i.e., the time-complexity of one step of an iterative algorithm), which is typically very small, as we see examples in Section \ref{sec:exam}.
That is, the algorithm has low complexity. In addition, throughput-optimality is achieved by a proper choice of weight vector $\bW(t)$ as a function of queue length vector $\bQ(t)$. We ensure that when $\bQ(t)$ is large, $\bW(t)$ does not change for sufficient amount of time {so that} the oracle system returns a maximum-weight schedule with respect to $\bW(t)$. This guarantees that our scheduling algorithms are throughput optimal.

\begin{figure}[ht!]
\centering
\includegraphics[width=3in]{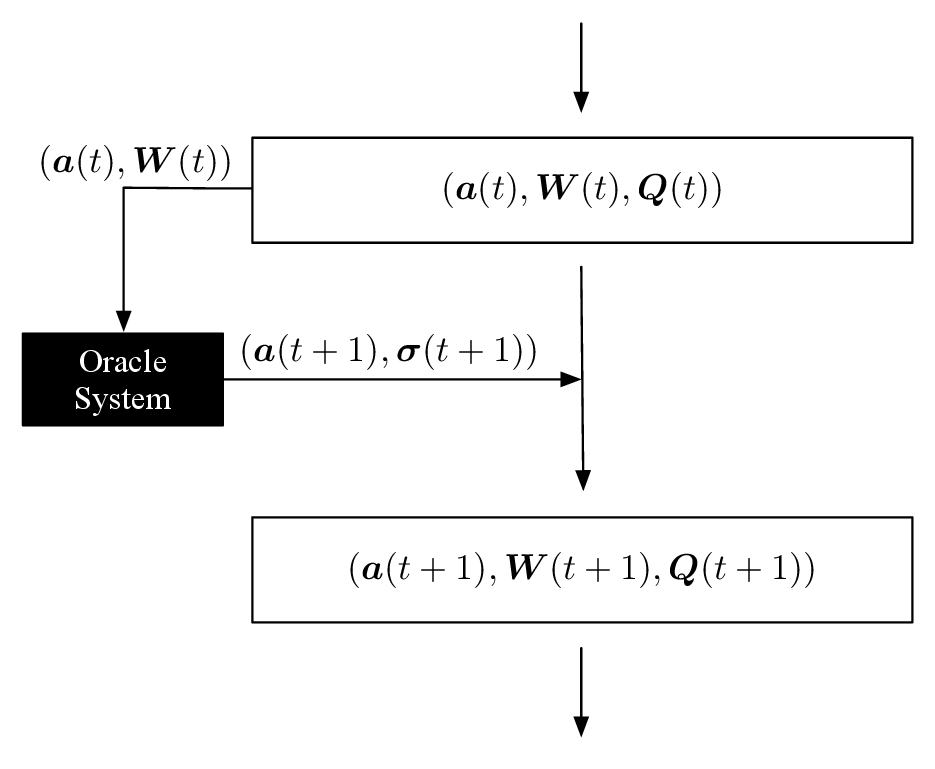}
\caption{\small Scheduling with an interactive oracle system. At each service epoch,
a query consisting of current advice $\ba(t)$ and weight $\bW(t)$ is sent to the oracle system. The oracle returns {updated advice $\ba(t+1)$ and 
schedule $\bsigma(t+1)$}.
}\label{fig:oracle}
\end{figure}

Next, we explain how to define $\bW(t)$. For each $i\in\mathcal{I}$, we let $W_i(t)$ be an integers in the interval $[U_i(t)-2, U_i(t)+2]$ for $i\in[n]$, where 
$$U_i(t):=\max\{f(Q_i(t)),g(Q_{\max}(t))\}$$
for positive real-valued functions $f,g:\mathbb R_+\to \mathbb R_+$ and $Q_{\max}(t)=\max_{i\in \mathcal{I}} Q_i(t)$.
At $t=0$, we define $W_i(0)$ be the closest integer to $U_i(0)$ and renew $W_i(t+1)$ for $t\geq 0$ as follows: For $i\in\mathcal{I}$ such that the distance between previous weight $W_i(t)$ and $U_i(t+1)$ is at least $2$, $W_i(t+1)$ becomes the closest integer to $U_i(t+1)$, and $W_i(t+1)$ is the same as $W_i(t)$ for the other $i$'s. The following is a formal description of the procedure at each service time.

\vspace{0.2cm}

\noindent\rule{\linewidth}{0.4mm}
\begin{itemize}
  \item[$\circ$] $\bsigma(t+1) = \bsigma_{\oracle}(\bW(t),\ba(t))$,
	\item[$\circ$] $\ba(t+1) = \ba_{\oracle}(\bW(t),\ba(t))$,
	\item[$\circ$] $W_i(t+1)$ is the closest integer to $U_i(t+1)$ if $$|W_i(t)-U_i(t+1)|> 2,$$ and 
	$W_i(t+1)=W_i(t)$ otherwise.
\end{itemize}
\noindent\rule{\linewidth}{0.4mm}
\vspace{0.1cm}

{Now, we are ready to state our main theorem, which introduces the sufficient condition for functions $f$ and $g$ to guarantee throughput-optimality of the algorithms.} 

\begin{theorem}\label{thm:mainResult} 
The above scheduling algorithm is throughput-optimal
if functions $f,g,h$ satisfy 
condition {\bf C0} in addition to the following conditions:
	\begin{enumerate}
    \item[\bf C1.] $f$ and $g$ are increasing, differentiable, and concave. 
    \item[\bf C2.] $\lim_{x\to\infty} \frac{g(x)}{f(x)} = 0$,  and $\lim_{x\to\infty} g(x) = \infty$. 
    \item[\bf C3.] $f(0)=0$.
    \item[\bf C4.] $\lim_{x\to\infty} f'(x)=\lim_{n\to\infty} g'(x)=0$.
    \item[\bf C5.] For any fixed $\eta,\delta>0$,
    \begin{equation*}
      \lim_{x\to\infty} \frac{h(f(x),\eta,\delta)}{x} =0.
    \end{equation*}
    \item[\bf C6.] There exists $c\in(0,1)$ such that for any fixed $\eta,\delta>0$, 
    \begin{eqnarray*}
      && \lim_{x\to\infty} f^{\prime}\left(f^{-1}\left(g\left((1-c)x\right)\right)\right) h\left(f((1+c)x),\eta,\delta\right)=0.
    \end{eqnarray*}
  \end{enumerate}
\end{theorem}
\noindent
We {provide} some intuitions underlying the above conditions.
Conditions {\bf C1}, {\bf C3} and {\bf C4} are technical conditions that make
our analysis using a Lyapunov function easier. 
{Condition {\bf C2} implies that $f$ should grow faster than $g$.
Therefore, weight $W_i(t)\approx U_i(t)=\max\{f(Q_i(t)),g(Q_{\max}(t))\}$ is determined by $f$ and $g$ 
for large and small queue $Q_i(t)$, respectively.}
To establish throughput-optimality,
we prove that
if the maximum queue length $Q_{\max}(t)$ is large,
weight function $W_i(t)$ 
remains constant for long enough stretches of time
so that
the interactive oracle produces an approximation solution of \eqref{eq:max_problem}, i.e.,
achieves the maximum weight schedule.
To this end, we need the property that $U_i(t)$ changes slowly,
where conditions {\bf C5} and {\bf C6} ensure it for maximum and non-maximum queues, respectively,
as explained in what follows.
From Condition {\bf C5}, 
$f$ should grow slowly with respect to $h$, i.e.,
$U_i(t)=f(Q_{\max}(t))$ for maximum queues change slowly.
The change of $U_i(t)$ for other non-maximum queues is
larger than that for maximum queues, but
the term $f^{\prime}\left(f^{-1}\left(g\left((1-c)x\right)\right)\right)$ in condition {\bf C6}
will be used to
bound the change of $U_i(t)$ for non-maximum queues.
Namely, 
condition {\bf C6} is necessary to guarantee that
$U_i(t)$ for non-maximum queues changes slowly with respect to $h$.
Note that due to condition {\bf C6}, $g$ should grow ``not too slowly''.

Our proof formalizes the above intuitions.
The proof outline of the above theorem is presented in the following section, 
and detailed proofs of key lemmas are given in Section \ref{sec:pflemmas}.
In Section \ref{sec:exam}, we present several specific examples of throughput-optimal and low-complexity scheduling
algorithms under Theorem \ref{thm:mainResult}. 

\subsection{Proof Outline of Theorem \ref{thm:mainResult}}\label{sec:pfmainthm}

We will utilize Lemma \ref{thm:Lyapunov} to show the desired throughput-optimality. 
To this end, we first define a {Markov chain describing the evolution of the network system}.
Under our scheduling algorithm, at time $t$, we retain advice $\ba(t)$, weight vector $\bW(t)$, and queue length vector $\bQ(t)$, all of which depend  on only the previous ones: $\ba(t-1)$, $\bW(t-1)$, and $\bQ(t-1)$. Therefore, $$\{\bX(t):=(\ba(t), \bW(t), \bQ(t))\,:\,t\in\mathbb Z_+\}$$ is a Markov chain on the state space
{
\begin{equation*}
  \Omega:=\bigg\{ (\ba,\bW,\bQ)\in\mathcal{A} \times \mathbb Z_+^n \times \mathbb Z_+^n    \,:\, 
  |W_i-U_i| \leq 2, 
   \textrm{where } U_i=\max\{f(Q_i),g(Q_{\max})\}
  \bigg\}.
\end{equation*} 
}
For $\bx=(\ba,\bW,\bQ)\in\Omega$, we consider the following Lyapunov function:
\begin{equation*}
	L(\bx)~:=~\sum_{i=1}^n \int_0^{Q_i}f(s)d s.
\end{equation*}
Since $\lim_{x\to\infty} f(x)=\infty$ (i.e., condition {\bf C2}), we have that $\sup_{\bx\in\Omega}L(\bx)=\infty$
and $L$ is bounded if and only if queue lengths are bounded. 
Therefore, 
the positive recurrence of $\mathcal{B}_\gamma=\{\bx\in\Omega \,:\,L(\bx)\leq\gamma\}$ for large enough $\gamma$ guarantees the stability of the system, i.e., 
queue lengths remain finite with probability 1.

To establish the positive recurrence of $\mathcal{B}_\gamma$, we define functions $\tau,\kappa:\Omega\to\mathbb R_+$ that satisfy \eqref{eq:thm:Lyapunov} and conditions {\bf L1}--{\bf L4} in Lemma~\ref{thm:Lyapunov} when $\blambda\in\Lambda^o$.
First, observe that for any $\blambda \in \Lambda^o$,
there exists $\varepsilon>0$ and $[\alpha_{\brho}:\brho\in S]\in [0,1]^{|S|}$ so that
\begin{equation}
  \sum_{\brho \in S} \alpha_{\brho}=1-\varepsilon <1 \quad\mbox{and}\quad \blambda<\sum_{\brho\in S}\alpha_{\brho}\brho. \label{eq:lambda_epsilon}
\end{equation}
For state $\bx=(\ba,\bW,\bQ)\in\Omega$, we define
\begin{eqnarray}
   \tau(\bx)&=&\left\lfloor \frac{1}{(n+\mu\sqrt{n})+1}\,\min\left\{ \frac{1}{f'\left(f^{-1}(g((1-c)Q_{\max}) ))\right)},~ c\, Q_{\max}  \right\} \right\rfloor,\label{eq:deftau}    \\
   \frac{\kappa(\bx)}{\tau(\bx)}
	&=& \left( \frac{\varepsilon}{2}(1-\alpha)(1-\beta)+\frac{2n}{1-c}\big((1-\beta)\alpha+\beta\big) \right)f((1-c)Q_{\max}) \nonumber\\
	&&\mbox{}~-~ n\left( \frac{f(Q_{\max})}{\tau(\bx)} +(\mu^2+2)f'(0)+n+\mu\sqrt{n}+1 \right) \label{eq:defkappa},
\end{eqnarray}
where $\mu^2$ is an upper bound of variance of $A_i(t)$, $c$ is the constant appearing in condition {\bf C6}, $\lfloor x \rfloor$ the largest integer not greater than $x$,
and $\alpha$, $\beta\in(0,1)$ constants satisfying 
\begin{equation*}
  \frac{\varepsilon}{2}(1-\beta)(1-\alpha)-\frac{2n(\beta+(1-\beta)\alpha)}{1-c}>0.
\end{equation*}
For example, one can choose $\alpha=\beta = \frac{\varepsilon(1-c)}{32n}$.
Using the above functions, we establish the following lemma, the proof of which is presented in
Section~\ref{sec:pflemmas}.
\begin{lemma} \label{lemma:negative_drift}
  Given arrival rate vector $\blambda \in \Lambda^o$ and initial state $\bx=(\ba, \bW, \bQ)\in\Omega$ with large enough $Q_{\max}:=\max_{i\in \mathcal{I}} Q_i$, we have
  \begin{equation}
    \E\big[L(\bX(\tau(\bx)))-L(0)\,|\,\bX(0)=\bx\big]~\leq~ -\kappa(\bx).\label{eq:negative_drift}
  \end{equation}
\end{lemma}

\noindent
We explain why we define $\tau(\bx)$ and $\kappa(\bx)$ as in \eqref{eq:deftau} and \eqref{eq:defkappa}, respectively,
in Section \ref{sec:pflemmas}. In essence, we define $\tau(\bx)$ large enough so that
the weights of schedules are close to the maximum weight mostly in the time interval $[0,\tau(\bx)]$.
The definition \eqref{eq:defkappa} of $\kappa(\bx)/\tau(\bx)$
consists of the first positive and second negative terms.
If the weights of schedules are close to the maximum weight, the negative draft of $L$ occurs, which
contributes the first positive term of \eqref{eq:defkappa}.
The second negative term of \eqref{eq:defkappa} bounds the possible positive draft of $L$
for other cases.
Moreover, from Lemma \ref{lemma:negative_drift}, without loss of generality, one can assume that \eqref{eq:negative_drift} holds for every $\bx\in\Omega$ (i.e.,
\eqref{eq:thm:Lyapunov} of Lemma \ref{thm:Lyapunov} holds):
if it does not hold for $\bx$ with {small $Q_{\max}$}, one can redefine
$\tau(\bx)=\kappa(\bx)=0$ for those cases, and this redefining does not affect the following arguments that verify $B_\gamma$ is positive recurrent.

Now, we check that 
$\tau$ and $\kappa$ satisfy conditions {\bf L1}--{\bf L4}
of Lemma~\ref{thm:Lyapunov}. 
Toward this,
we investigate limits of $\tau(\bx)$ and $\kappa(\bx)/\tau(\bx)$ as $L(\bx)\to\infty$: 
\begin{eqnarray}
  \lim_{L(\bx)\to\infty} \tau(\bx)&=&\infty  \label{eq:tau_infinity}\\
  \lim_{L(\bx)\to\infty} \kappa(\bx)/\tau(\bx)&=&\infty, \label{eq:kappa_infinity}
\end{eqnarray}
the proof of which are elementary and given in Appendix~\ref{app:infty} for completeness. The above two equations imply that 
\begin{equation}\label{eq:limit_kappa}
    \lim_{L(\bx)\to\infty}\kappa(\bx)~=~\infty
\end{equation}
which verifies condition {\bf L1} (i.e., $\inf_{L(\bx)\to\infty} \kappa(\bx)>0$). In addition, since $\kappa,\tau$ are bounded as long as $L$ is bounded,
condition {\bf L3} (i.e., $\sup_{x\in B_{\gamma}}\tau(\bx)<\infty$) follows and
\eqref{eq:limit_kappa} implies condition {\bf L2} (i.e., 
$\inf_{\bx\in\Omega}\kappa(\bx)>-\infty$). 
Finally, \eqref{eq:kappa_infinity} implies condition {\bf L4} (i.e., $\limsup_{L(\bx)\to\infty} \tau(\bx)/\kappa(\bx)<\infty$).
This completes the proof of Theorem~\ref{thm:mainResult}.

\section{Applications}\label{sec:exam}
This section shows the wide applicability of our framework 
{by illustrating several throughput-optimal and low-complexity scheduling algorithms interacting with various oracle systems.}
As we mentioned in Section~\ref{sec:oracle_system}, oracle systems are derived from iterative methods for solving optimization problem \eqref{eq:max_problem}:
\begin{equation*} 
  \max \left\{ \brho\cdot\bW:=\sum_{i\in \mathcal{I}} \rho_i W_i\,:\,\brho\in \mathcal{S}\right\}
\end{equation*}
and such methods depend on the underlying structure of constrained queueing network $(\mathcal{I},\mathcal{S})$.
Thus, to illustrate an oracle system from an iterative method, we begin by introducing specific network systems in which the method finds an approximate solution to \eqref{eq:max_problem} with high probability. Then, we 
{construct} the oracle system by identifying advice space $\mathcal{A}$, inputs, and outputs, in addition
to finding function $h$ that satisfies condition~{\bf C0}. Finally, we provide explicit functions $f$ and $g$ and prove that they satisfy conditions {\bf C1}--{\bf C6} of Theorem~\ref{thm:mainResult}, {from which the throughput-optimality of the scheduling algorithm immediately follows as a corollary.}

\subsection{Random Search (RS): Pick-and-Compare}\label{sec:ES}
The first oracle system that we introduce utilizes the naive random search (RS) method, which 
maintains a current schedule $\bsigma\in \mathcal{S}$. At each iteration, the method picks a new schedule $\brho\in\{0,1\}^n$ uniformly at random and, if $\brho$ is in $\mathcal{S}$ and the weight of $\brho$ is greater than that of $\bsigma$, $\bsigma$ is replaced by $\brho$. 
Now, we formally describe the oracle system called \emph{RS oracle system}. 

\smallskip
\noindent{\bf RS oracle system.}
The advice space of the RS oracle system is $\mathcal{S}$ (i.e., $\mathcal{A}=\mathcal{S}$).
When the oracle system receives advice $\ba=\bsigma\in \mathcal A$ and weight vector $\bW\in\mathbb Z_+$ as inputs, 
it returns $\bsigma_{\oracle}(\bW, \ba)=\widehat{\bsigma}$ and $\ba_{\oracle}(\bW, \ba)=\widehat{\bsigma}$ obtained as follows:
\vspace{0.2cm}

\noindent\rule{\linewidth}{0.4mm}
\begin{itemize}
\item[1.] Pick $\brho\in\{0,1\}^n$ uniformly at random.
  \item[2.] Set
  $\widehat \bsigma=\begin{cases}
    \brho & \textrm{if $\brho\in S$ and $\brho \cdot \bW > \bsigma \cdot \bW$}\\
    \bsigma & \textrm{otherwise}
  \end{cases}.$
\end{itemize}

\noindent\rule{\linewidth}{0.4mm}
\vspace{0.1cm}
\noindent
At each query, the oracle system returns a maximum-weight schedule with a probability of at least ${1}/{2^n}$, so function $h$ in condition {\bf C0} can be defined as 
\begin{equation}
	h(W_{\max},\eta,\delta)~:=~\frac{\log\delta}{\log\left(1-1/2^n\right)},\label{eq:h_random_search}
\end{equation}
which is independent of weight $\bW$. 
The following corollary shows how we choose functions $f$ and $g$ to guarantee the throughput-optimality of the scheduling algorithm with the RS oracle system.
\begin{corollary}\label{cor:es}
The scheduling algorithm described in Section \ref{sec:algorithm_description} using the 
RS oracle system is throughput-optimal if
$$f(x) = x^a,~g(x) = x^b, ~\mbox{and}~
0<b<a<1.$$
\end{corollary}
\begin{proof}
It is elementary to check conditions {\bf C1}--{\bf C5} of Theorem \ref{thm:mainResult} for $h(W_{\max},\eta,\delta)$ in \eqref{eq:h_random_search}. Condition {\bf C6} of Theorem \ref{thm:mainResult} can be derived as follows: for $0<c<1$,
  \begin{eqnarray*}
    \lim_{x\to\infty} f^{\prime}\left(f^{-1}\left(g\left((1-c)x\right)\right)\right) h\left(f((1+c)x),\eta,\delta\right)
    ~=~ \lim_{x\to\infty}  \frac{\log\delta}{\log\left(1-1/2^n\right)}  \,a\,\left((1-c)\right)^{\frac{b(a-1)}{a}}\,x^{\frac{b(a-1)}{a}}~=~0.
  \end{eqnarray*}
Since functions $f$ and $g$ satisfy conditions {\bf C1}--{\bf C6}, the scheduling algorithm with the RS oracle system is throughput optimal according to Theorem~\ref{thm:mainResult}.
\end{proof}

\subsection{Markov Chain Monte Carlo (MCMC)}\label{sec:mcmc_wireless}
The second oracle system comes from the Markov chain Monte Carlo (MCMC) method, which solves the optimization problem~\eqref{eq:max_problem} 
for the following interference model in wireless networks. 

\smallskip
\noindent{\bf Wireless network model.} 
An interference model in a wireless network is represented by an undirected graph $G=(\mathcal V,\mathcal E)$ with $|\mathcal V|=n$ (e.g., see \cite{RSS09, SS12}).
$\mathcal{V}$ represents the set of links or queues (i.e., $\mathcal I=\mathcal V$), and they share an edge if they cannot transmit their packets simultaneously. Therefore,
the set of all available schedules $\mathcal S$ is defined as
\begin{equation}
\mathcal S~=~\big\{\bsigma\in\{0,1\}^n~:~\sigma_i+\sigma_j\leq 1,~~\forall~(i,j)\in \mathcal E\big\}.
\end{equation}
For buffer $i\in\mathcal{I}$, we define neighborhood $\mathcal{N}(i)$ as the set of buffers, which cannot transmit packets when buffer $i$ processes a packet:
$\mathcal{N}(i):=\{j\in\mathcal{I}\,:\,(i,j)\in\mathcal{E}\}$.
Figure~\ref{fig:wireless} illustrates a wireless network in a grid interference topology with nine buffers.
\begin{figure}[ht!]
\centering
\includegraphics[width=2.5in]{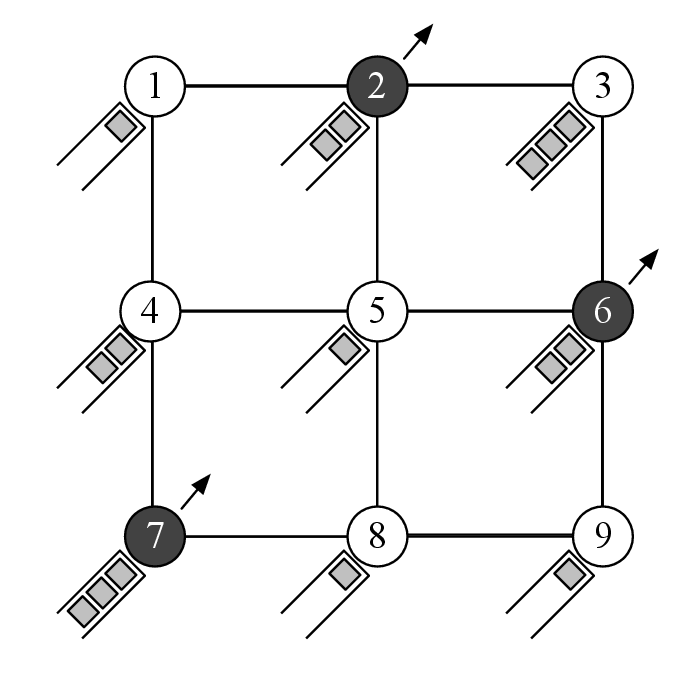}\caption{\small
Wireless network with $n=9$ queues in a grid interference topology.
Available schedules are $\{1,3,5,7,9\}$, $\{1,3,8\}$, $\{2,4,6,8\}$, $\{2,4,9\}$, $\{2,6,7\}$, $\{2,7,9\}$, and so on.
In this example, $\mathcal{N}(4)=\{1,5,7\}$.}\label{fig:wireless}
\end{figure}
 
\smallskip
\noindent{\bf MCMC oracle system.}
In the MCMC oracle system, the advice space is $\mathcal A=S$.
If the oracle system receives advice $\ba=\bsigma$ and weight vector $\bW$ as inputs, 
it returns $\bsigma_{\oracle}(\bW, \ba)=\widehat{\bsigma}$ and $\ba_{\oracle}(\bW, \ba)=\widehat{\bsigma}$, obtained from the following procedure:
\vspace{0.2cm}

\noindent\rule{\linewidth}{0.4mm}
\begin{enumerate}
  \item[1.] Choose buffer $i\in\mathcal{I}$ uniformly at random, and set 
$$\widehat \sigma_j=\sigma_j,\qquad\mbox{for all}~j\neq i.$$ 
  \item[2.] If $\sigma_j=1$ for some $j\in \mathcal{N}(i)$, then set $\widehat \sigma_i=0$.
\item[3.] 
Otherwise, set
\begin{equation*}
      \widehat{\sigma}_i=\begin{cases}
        1 & \mbox{with probability $\frac{\exp(W_i)}{1+\exp(W_i)}$}\\
        0 & \mbox{otherwise}
      \end{cases}.
    \end{equation*}
\end{enumerate}
\noindent\rule{\linewidth}{0.4mm}
\vspace{0.1cm}

\noindent
Then, existing results relating to the mixing time of MCMC show that condition {\bf C0} holds with
\begin{equation}\label{eq:hmcmc}
  h(W_{\max},\eta,\delta)= e^{C_1\,W_{\max}}\left(C_2+\log\left(\frac1{\eta\delta}\right)\right),
\end{equation}
where $C_1=C_1(n),C_2=C_2(n)$ are some (``$n$-dependent'') constants independent of $W_{\max}$. 
The proof of \eqref{eq:hmcmc} is a direct consequence of Lemmas 3 and 7 in \cite{SS12}, and 
we omit  the details because of space constraints. 
We can select functions $f$ and $g$ according to the following corollary so that the scheduling algorithm with the MCMC oracle system is throughput optimal.
\begin{corollary}\label{cor:mcmc}
  The scheduling algorithm described in Section \ref{sec:algorithm_description} using the MCMC oracle system is throughput-optimal if
  \begin{equation*}
    f(x)=(\log(x+e))^a-1\quad\mbox{and}\quad g(x)=(\log(x+e))^b-1,
  \end{equation*}
  where $0<a^2<b<a<1$.
\end{corollary}
\begin{proof}
  It is elementary to check conditions {\bf C1}--{\bf C4} of Theorem \ref{thm:mainResult}.
  Condition {\bf C5} is from \eqref{eq:hmcmc} and $f(x)=(\log(x+e))^a-1$: 
  \begin{align*}
     \frac{h(f(x),\eta,\delta)}{x} 
    ~=~ \left( C_2+ \log\left( \frac{1}{\eta\delta}\right)\right)
    \times e^{C_1\left(\log(x+e)\right)^a-\log x- C_1}~\stackrel{{x\to\infty}}{\to}~0.
  \end{align*}
  Furthermore,  condition {\bf C6} can be derived as follows:
  \begin{eqnarray*}
    \lefteqn{f'(f^{-1}(g((1-c)x)))\,h(f((1+c)x))} \\
    &=& \frac{a\left( C_2+ \log\left( 1/(\eta\delta)\right)\right)}
    {\left(\log((1-c)x+e)^{\frac{(1-a)b}{a}}\right)}
    \times
    e^{C_1\left(\log((1+c)x+e)\right)^a-\log((1-c)x+e)^{\frac{b}{a}}- C_1}
~\stackrel{{x\to\infty}}{\to}~ 0.
  \end{eqnarray*}
  This completes the proof of Corollary~\ref{cor:mcmc}.
\end{proof}

\noindent
We note that the scheduling algorithm described in Section \ref{sec:algorithm_description} using the MCMC oracle system
is a discrete-time version of the CSMA algorithm in \cite{RSS09, SS12}.

\subsection{Belief Propagation (BP)}\label{sec:bp} 
We derive the third oracle system from the belief propagation (BP) method, a popular heuristic iterative method for solving inference problems arising in probabilistic graphical models \cite{J04}.
For the provable throughput-optimality of the scheduling algorithm with the BP oracle system, 
we introduce a special constrained queueing network, called {\em input-queued switch model} \cite{KGL02}.

\smallskip
\noindent{\bf Input-queued switch model.}
An input-queued switch consists of $m$ input ports and $m$ output ports. 
An input port has $m$ buffers each of which stores packets to an output port. Thus, the total number of buffers in the system is $n=m^2$. 
Scheduling constraints in the input-queued switch as follows:
\begin{enumerate}
  \item Every input port can transmit at most one packet.
  \item Every output port can receive at most one packet.
\end{enumerate}
When an output port receives a packet, the packet leaves the system. We represent the above input-queue switch as an undirected complete bipartite graph of left vertices $\mathcal L$, right vertices $\mathcal R$, and edges $\mathcal E= \{(i,j)\,:\,i\in L,\,j\in R\}$, where $|\mathcal L|=|\mathcal R|=m$.
Then, each buffer is dented by $(i,j)\in \mathcal E$, so $\mathcal{I}=\mathcal{E}$. The set of all possible schedules is
\smallskip
\begin{equation}
\mathcal S=\left\{\bsigma\in\{0,1\}^{\mathcal E}~:~\begin{array}{l}
  \sum_{j:(i,j)\in \mathcal{E}}\sigma_{ij}\leq 1~\forall i\in \mathcal R,\\
  \sum_{i:(i,j)\in \mathcal{E}}\sigma_{ij}\leq 1~\forall j\in \mathcal L
\end{array}
\right\}. \label{eq:S_inputq}
\end{equation}
\smallskip
One can observe that this model is 
a special case
of the wireless network model described in the previous section.

\begin{figure}[ht!]
\centering
\includegraphics[width=2.5in]{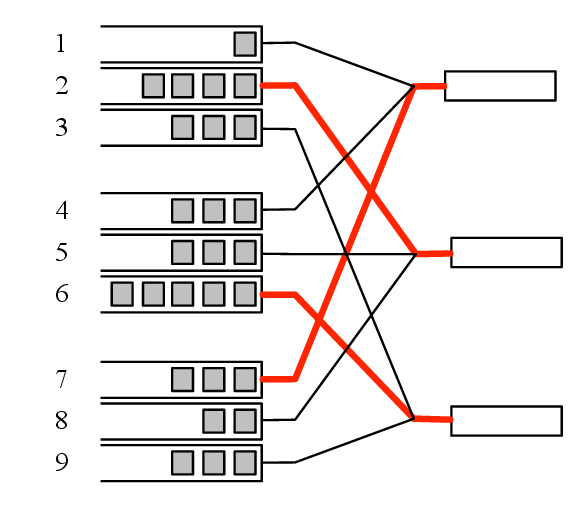}\caption{\small
Input-queued switch with $m=3$ input ports, $m=3$ output ports, and $m^2=9$ buffers. Available schedules are 
$\{1,5,9\}$, $\{1,6,8\}$, $\{2,4,9\}$, $\{2,6,7\}$, $\{3,4,8\}$, and $\{3,5,7\}$.}\label{fig:switch}
\end{figure}

\smallskip
\noindent{\bf BP oracle system.}
In the BP oracle system, the advice space is $\mathcal A = \mathbb Z_+^{2|\mathcal E|}\times \mathcal S$.
For the inputs of  weight vector $\bW$ and advice $\ba = (\bm, \bsigma)\in \mathcal A$, where $\bm = [m_{i\to j}, m_{j\to i}:(i,j)\in \mathcal E]$,
the oracle system outputs $\bsigma_{\oracle}(\bW, \ba)=\widehat{\bsigma}$ and $\ba_{\oracle}(\bW, \ba)=(\widehat \bm, \widehat{\bsigma})$
 calculated as follows:
\vspace{0.2cm}

\noindent\rule{\linewidth}{0.4mm}
\begin{itemize}
\item[1.] For each $(i,j)\in  \mathcal E$, set
\begin{eqnarray*}
  \widehat \sigma_{(i,j)}&=&\begin{cases}
    0 & \textrm{if $m_{i\to j}+m_{j\to i} >  W_{(i,j)}^{\prime}$}\\
    1 & \textrm{otherwise}
  \end{cases}\\
\widehat{m}_{i\to j}&=&\max_{k\neq j:(i,k)\in\mathcal  E} \left(W^{\prime}_{(i,k)} - m_{k\to i}\right)_+, 
\end{eqnarray*}
where 
$$W_{(i,j)}^{\prime} := W_{(i,j)}+ r_{(i,j)}
\quad\textrm{and}\quad
(x)_+:=\begin{cases}
x ~\mbox{if}~ x\geq 0\\
0~\mbox{otherwise}
\end{cases}.$$ 
\item[2.] If $\widehat \bsigma \notin \mathcal S$, reset $\widehat \bsigma = \bsigma$.
\end{itemize}
\noindent\rule{\linewidth}{0.4mm}
\vspace{0.2cm}
\noindent
In the above procedure, we need to choose $[r_{(i,j)}]\in [0,1]^{|\mathcal E|}$ such that 
$\bsigma^*\in \argmax_{\bsigma}\bsigma \cdot \bW^{\prime}
=\argmax_{\bsigma}\bsigma \cdot \bW$
is unique, and $\xi\leq \bsigma^*\cdot \bW^{\prime} - \max_{\bsigma\neq \bsigma^*} \bsigma\cdot \bW^{\prime}$
for some constant $\xi>0$. For example, we can set 
$$\mbox{$r_{(i,j)}=\frac1{2^i 2^{m+j}}$,$\quad$ where $i,j\in\{1,\dots, m\}$.}$$
Then, 
from work by Bayati et al.\ \cite{BSS08} and Sanghavi et al.\ \cite{SMW07}, condition {\bf C0} holds with 
$$h=h(W_{\max},\eta,\delta)=O(W_{\max}/\xi).$$
The following corollary suggests to the choice of functions $f$ and $g$ so that the scheduling algorithm with the BP oracle system is throughput optimal.

\begin{corollary}\label{cor:bp}
The scheduling algorithm described in Section \ref{sec:algorithm_description} using the BP oracle system is throughput-optimal if
\begin{equation*}
  f(x) = x^a,\ g(x) = x^b,\ \mbox{and}\ 0<\frac{a^2}{1-a}<b<a<\frac12.
\end{equation*}
\end{corollary}
\begin{proof}
It is elementary to check conditions {\bf C1}--{\bf C5} of Theorem \ref{thm:mainResult}, where $h(W_{\max},\eta,\delta)=O(W_{\max}/\xi)$. 
  Condition {\bf C6} of Theorem \ref{thm:mainResult} can be derived as follows: for $c>0$, 
  \begin{eqnarray*}
    \lefteqn{\lim_{x\to\infty} f^{\prime}\left(f^{-1}\left(g\left((1-c)x\right)\right)\right) h\left(f((1+c)x),\eta,\delta\right)}\\
    &=& \lim_{x\to\infty} 
    C\cdot x^{\frac{b(a-1)}{a}}\left((1+c)^a x^a+1\right)~=~ 0,
  \end{eqnarray*}
  where $C$ is some constant depending on $\xi, c,a,b$ and the last equality is from $0<b,a<1$ and $b>\frac{a^2}{1-a}$. 
This completes the proof of Corollary \ref{cor:bp}.
\end{proof}

We also note that one can design the BP oracle in various ways, one of which is the following:

\noindent\rule{\linewidth}{0.4mm}
\begin{itemize}
\item[1.] For each $(i,j)\in \mathcal E$, set
\begin{eqnarray*}
\widehat \sigma_{(i,j)}&=&0\\
  b_{(i,j)}&=& W_{(i,j)}^{\prime}- m_{i\to j}-m_{j\to i}  \\
\widehat{m}_{i\to j}&=&\max_{k\neq j:(i,k)\in \mathcal E} \left(W^{\prime}_{(i,k)} - m_{k\to i}\right)_+.
\end{eqnarray*}
\item[2.] Choose $(i,j)\in E$ so that $b_{(i,j)}$ is the largest among  
those which $\widehat \bsigma \in S$ after resetting $\widehat \sigma_{(i,j)}=1$.
Reset $\widehat \sigma_{(i,j)}=1$, and keep this ``greedy'' procedure until no edge is found.
\end{itemize}

\noindent\rule{\linewidth}{0.4mm}
\vspace{0.05cm}

\noindent
While the first BP oracle system simply checks whether the ``belief'', $b_{(i,j)}$, is positive or not, the second BP oracle system determines schedule $\widehat\bsigma(t)$ greedily based on $[b_{(i,j)}]$.
When we use the same set of functions $f$ and $g$ in Corollary~\ref{cor:bp}, the scheduling algorithm with the above second BP oracle system is also throughput optimal, and the proof is identical to that of the first BP oracle system. 
We note that a similar version of the scheduling algorithm using the second oracle system was first studied in \cite{SDPM10} heuristically, but
our results (Theorem~\ref{thm:mainResult}) provide its formal throughput-optimality proof, which is missing in \cite{SDPM10}.

\subsection{Primal-Dual Method (PDM)}\label{sec:cp_wireless}
We introduce the fourth oracle system, called the primal-dual method (PDM). 
For a detailed description of the oracle, we first introduce a primary interference constrained wireless model.

\smallskip
\noindent{\bf Primary interference constrained wireless network model.} This network model is represented by a directed graph, $G=(\mathcal V,\mathcal E)$ with $|\mathcal E|=n$, and the set of available schedules $\mathcal S$ is defined as
\begin{equation}
\mathcal S~=~\left\{\bsigma\in\{0,1\}^{\mathcal E}~:~\sum_{j:(j,i)\in \mathcal E}\sigma_{ji}\leq 1,~~\sum_{j:(i,j)\in \mathcal E}\sigma_{ij}\leq 1,~~\forall~i\in \mathcal V\right\}.
\end{equation}
The above ``matching'' scheduling constraint has been popularly used for modeling primary interference in wireless networks \cite{SBS07}, which is also a
 special case of the wireless network model in Section \ref{sec:mcmc_wireless}.

\begin{figure}[ht!]\label{fig:primary}
\centering
\includegraphics[width=2.5in]{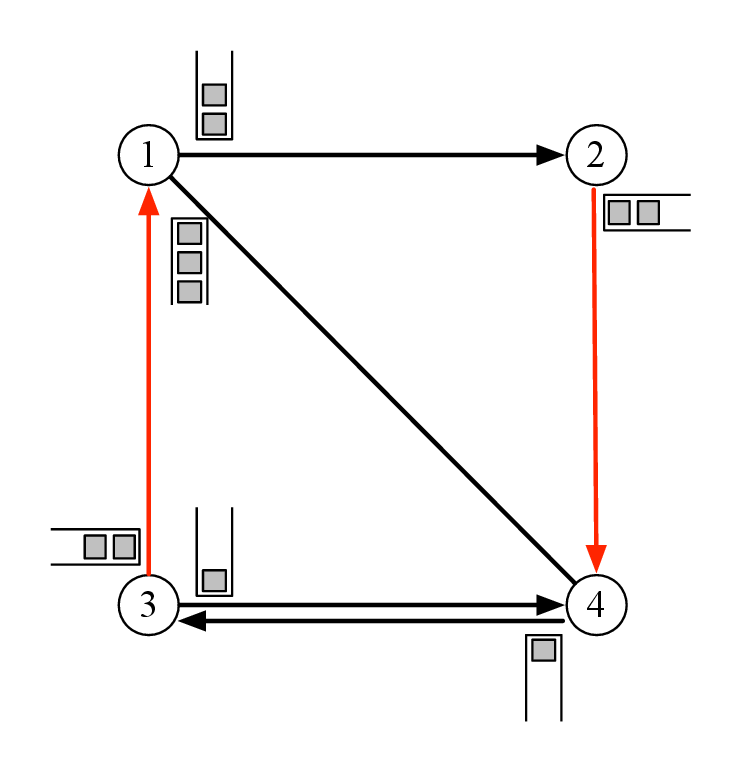}\caption{\small
Primary wireless network. Available schedules are $\{ 1\to 2, 3\to 4\},
\{1\to 4\},\{2\to 4,3\to 1 \},$ and so on.}
\end{figure}

\smallskip
\noindent{\bf PDM oracle system.}
The PDM method is an iterative mechanism introduced by Edmonds \cite{E65,E65-1} which 
maintains primal and dual variables of a Linear Programming (LP), 
and updates them until the primal solution $\mathbf x\in \mathcal S$ reaches the maximum weight schedule (i.e., matching).
At each iteration, the primal solution always forms a matching and the dual solution $\by$ is feasible, where each edge in the primal matching
should be `tight' with respect to the dual solution (see the most recent implementation
of the Edmonds' algorithm by Kolmogorov \cite{K09} for more details).
Formally, in the PDM oracle system, advice space $\mathcal A$ is the set of primal and dual variables, and for advice $\ba=(\mathbf x, \mathbf y)\in\mathcal A$, 
the oracle outputs $\bsigma_{\oracle}(\bW, \ba)=\widehat{\bsigma}$ and $\ba_{\oracle}(\bW, \ba)=\widehat{\ba}$, which are chosen as follows:

\vspace{0.2cm}

\noindent\rule{\linewidth}{0.4mm}
\begin{itemize}
\item[1.] If the dual solution $\mathbf y$ is not feasible, make it feasible by re-normalizing.
  \item[2.] If some edge in the primal solution $\mathbf x$ is not tight with respect to the dual solution, remove it.
\item[3.] Obtain new primal and dual solutions $\widehat{\mathbf x}, \widehat{\mathbf y}$,
as described in \cite{K09}.
\item[4.] Set $\widehat{\ba}=(\widehat{\mathbf x}, \widehat{\mathbf y})$ and $\widehat \bsigma=\widehat{\mathbf x}$.
\end{itemize}

\noindent\rule{\linewidth}{0.4mm}
\vspace{0.1cm}
\noindent
It is well known that condition {\bf C0} holds with
$h=h(W_{\max},\eta,\delta)=O(n).$
Then, the following theorem suggests how to select functions $f$ and $g$ so that the scheduling algorithm with the PDM oracle system is throughput optimal.

\begin{corollary}\label{cor:cp}
The scheduling algorithm described in Section \ref{sec:algorithm_description} using the 
PDM oracle system is throughput-optimal if
$$f(x) = x^a,~g(x) = x^b, ~\mbox{and}~
0<b<a<1.$$
\end{corollary}
\begin{proof}
It is elementary to check conditions {\bf C1}--{\bf C5} of Theorem \ref{thm:mainResult}, for $h(W_{\max},\eta,\delta)=O(n)$. 
Condition {\bf C6} of Theorem \ref{thm:mainResult} can be derived as follows: for $0<c<1$, since $h$ is independent on $W_{\max}$, 
  \begin{eqnarray*}
    \lefteqn{\lim_{x\to\infty} f^{\prime}\left(f^{-1}\left(g\left((1-c)x\right)\right)\right) h\left(f((1+c)x),\eta,\delta\right)}\\
    &=& \lim_{x\to\infty}  h(x,\eta,\delta)  \,a\,\left((1-c)\right)^{\frac{b(a-1)}{a}}\,x^{\frac{b(a-1)}{a}}~=~0.
  \end{eqnarray*}
Because $f$ and $g$ satisfy all conditions in Theorem~\ref{thm:mainResult}, the scheduling algorithm with the PDM oracle system is throughput optimal.
\end{proof}

\section{Proof of Lemma~\ref{lemma:negative_drift}}\label{sec:pflemmas}
In this section, we show the following
negative drift property of $L$:
\begin{equation}
    \E\big[L(\bX(\tau(\bx)))-L(0)\,|\,\bX(0)=\bx\big]~\leq~ -\kappa(\bx),\label{eq:desired}
\end{equation}
for arrival rate vector $\blambda \in \Lambda^o$ and initial state $\bx=(\ba, \bW, \bQ)\in\Omega$ with large enough $Q_{\max}:=\max_{i\in \mathcal{I}} Q_i$. 

\smallskip
\noindent {\bf Proof outline.} The proof consists of several steps with associated lemmas and propositions.
{Before we detail the proof, we summarize our high-level strategy to prove Lemma \ref{lemma:negative_drift}.}

First,  
we introduce a random variable $\Delta(\bx)$ such that
\begin{equation*}
	\E\big[L(\bX(\tau(\bx)))-L(0)\,|\,\bX(0)=\bx\big]~\leq~\E\big[ \Delta(\bx)\,|\,\bX(0)=\bx \big]+O(1).
\end{equation*}
Then, we define an event $\mathcal{E}_1$ that occurs with high probability, and 
on the event, we obtain an upper bound of $\E\big[\Delta(\bx)\,|\,\bX(0)=\bx\big]$,
which is formally stated in Lemma~\ref{lemma:bounds_on_E1}.
This leads to 
the proof of 
the desired inequality \eqref{eq:desired}, where the definition \eqref{eq:defkappa}
of $\kappa(\bx)$ is used.
To prove Lemma \ref{lemma:bounds_on_E1},
we show that the weight $\bW(t)\approx f(\bQ(t)) $ does not change many times on $[0,\tau(\bx)]$ for large enough $Q_{\max}$ under $\mathcal{E}_1$,
which is formally stated in Proposition~\ref{prop:stopping_time}, and 
{in its proof, we define $\tau(\bx)$ appropriately as in \eqref{eq:deftau}.}
Since $\bW(t)$ remains fixed for long enough time on $[0,\tau(\bx)]$, the oracle satisfying condition {\bf C0} returns schedules with weights close to the maximum weight mostly in the time interval, which is formally stated in Proposition~\ref{prop:main_lemma}. 
This leads to the proof of Lemma \ref{lemma:bounds_on_E1}.

We provide the proofs of Lemma~\ref{lemma:bounds_on_E1},
Proposition~\ref{prop:stopping_time}, and Proposition \ref{prop:main_lemma} in 
Section \ref{sec:cond_exp_E}, Section~\ref{sec:pf_st}, and Section \ref{sec:pr_ber}, respectively.

\subsection{Proof of Lemma~\ref{lemma:negative_drift}}\label{sec:pf:lemma:negative_drift}
In this subsection, we provide the proof of Lemma \ref{lemma:negative_drift} apart from a key lemma, Lemma \ref{lemma:bounds_on_E1}.
For notational simplicity, we use $L(t)$ to denote $L(\bX(t))$.
We start with the following proposition, the proof of which is quite standard in the literature (e.g., see \cite{TE92}).
\begin{proposition}\label{prop:max_weight}
For the Markov chain $\{\bX(t)\,:\,t\in\mathbb Z_+\}$ defined in Section~\ref{sec:pfmainthm}, we have
\begin{eqnarray}
  L(t+1)-L(t) 
  &=& \sum_{i=1}^n \int_{Q_i(t)}^{Q_i(t+1)} f(s) d s \nonumber \\
  &\leq& \bA(t)\cdot f(\bQ(t))-\bsigma(t)\cdot f(\bQ(t))+
  f'(0)\left(\sum_{i=1}^n A_i(t)^2+n\right), \label{eq:general_bound_L}
\end{eqnarray} 
where $\bA(t)=[A_i(t):i\in \mathcal{I}]$ and $f(\bQ(t))=[f(Q_i(t)):i\in \mathcal{I}]$.
\end{proposition}
\begin{proof} It is sufficient to show that
\begin{equation}
  \int_{Q_i(t)}^{Q_i(t+1)} f(s) d s ~\leq~ f(Q_i(t))A_i(t)-f(Q_i(t))\sigma_i(t)+f'(0)(A_i^2(t)+1),\quad\forall i\in \mathcal{I}.\label{eq:bound_L_each}
\end{equation}
We verify \eqref{eq:bound_L_each} in cases: $Q_i(t+1)\geq Q_i(t)$ and $Q_i(t+1)<Q_i(t)$. 

\smallskip
First, assume that $Q_i(t+1)\geq Q_i(t)$. Since $f$ is convex and $f'$ is non-increasing, we obtain
\begin{equation*}
  f(s)~\leq~f(Q_i(t))+f'(Q_i(t))(s-Q_i(t))~\leq~f(Q_i(t))+f'(0)(s-Q_i(t))~\leq~f(Q_i(t))+f'(0)A_i(t),
\end{equation*}
for all $Q_i(t)\leq s\leq Q_i(t+1)$. Therefore, we conclude that
\begin{eqnarray*}
  \int_{Q_i(t)}^{Q_i(t+1)} f(s) d s
  &\leq& \big(f(Q_i(t))+f'(0)A_i(t)\big)\big(Q_i(t+1)-Q_i(t)\big)\nonumber \\
  &\leq& f(Q_i(t))\,A_i(t)-f(Q_i(t))\,\sigma_i(t)+f'(0)A_i(t)^2,
\end{eqnarray*}
which shows \eqref{eq:bound_L_each} holds when $Q_i(t+1)\geq Q_i(t)$.

\smallskip
Second, suppose that $Q_i(t+1)<Q_i(t)$. Then, because $f$ is convex and $f'$ is non-increasing, we have
\begin{equation*}
  f(Q_i(t))~\leq f(s)+f'(s)(Q_i(t)-s)~\leq~f(s)+f'(0)(Q_i(t)-s),\quad\forall s\in[Q_{i+1}(t),Q_i(t)],
\end{equation*}
so we obtain
\begin{equation*}
  -f(s)~\leq~ -f(Q_i(t))+f'(0)(Q_i(t)-s)~\leq~-f(Q_i(t))+f'(0),\quad\forall s\in[Q_{i+1}(t),Q_i(t)],
\end{equation*}
where we use $Q_{i+1}(t)\geq Q_i(t)-1$ for the last inequality.
This inequality implies that
\begin{eqnarray*}
  \int_{Q_i(t)}^{Q_i(t+1)} f(s) d s
  &=&\int_{Q_i(t+1)}^{Q_i(t)}-f(s) d s \nonumber \\
  &\leq& \int_{Q_i(t+1)}^{Q_i(t)}-f(Q_i(t))+f'(0) d s\nonumber \\
  &=& \int_{Q_i(t)}^{Q_i(t+1)} f(Q_i(t))-f'(0) d s \nonumber \\
  &\leq& f(Q_i(t))A_i(t)-f(Q_i(t))\sigma_i(t)+f'(0)
\end{eqnarray*}
where the last inequality follows from $Q_i(t+1)-Q_i(t)\geq -1$. This inequality verifies \eqref{eq:bound_L_each} for the case of $Q_i(t+1)<Q_i(t)$.
\end{proof}

When one takes expectation of \eqref{eq:general_bound_L}, the first term of right hand side becomes
\begin{eqnarray}
  \E_\bx[\bA(t)\cdot f(\bQ(t))]&:=& \E[\bA(t)\cdot f(\bQ(t))\,|\,\bX(0)=\bx] \nonumber \\[0.1in]
  &=& \E_\bx[\bA(t)]\cdot \E_\bx[f(\bQ(t))] \nonumber \\
  &\leq& \sum_{\brho \in S} \alpha_{\brho}\brho \cdot \E_\bx[f(\bQ(t))] 
  ~\leq~ (1-\varepsilon)\;\E_\bx\!\!\left[\max_{\brho\in S}\big\{\brho\cdot f(\bQ(t))\big\}\right], \label{eq:arrival_bound}
\end{eqnarray}
where the inequalities come from \eqref{eq:lambda_epsilon}. Here, to simplify notation, we use $\E_\bx[Y]$ to denote the conditional expectation of random variable $Y$ under the initial state $\{\bX(0)=\bx\}$.
Note that $\E[A_i(t)^2]=\var(A_i(t))+\E[A_i(t)]^2\leq \mu^2+1$. Then, by summing \eqref{eq:general_bound_L} from $t=0$ to $t=\tau(\bx)-1$ and applying \eqref{eq:arrival_bound}, we obtain
\begin{eqnarray}
\E_{\bx}\!\left[L(\tau(\bx))-L(0)\right]
&=&\E_\bx\!\!\left[\sum_{t=0}^{\tau(\bx)-1}L(t+1)-L(t)\right] \nonumber \\
&\leq& \sum_{t=0}^{\tau(\bx)-1}\E_\bx\!\left[ \bA(t)\cdot f(\bQ(t))-\bsigma(t)\cdot f(\bQ(t))\right]
 +n\left(\mu^2+2\right)f'(0)\tau(\bx) \nonumber \\[5pt]
&\leq& \E_\bx[\Delta(\bx)]+n\left(\mu^2+2\right)f'(0)\,\tau(\bx) \label{eq:L_Delta},
\end{eqnarray}
where 
$$
\Delta(\bx) := \sum_{t=0}^{\tau(\bx)-1}\left((1-\varepsilon)\,\max_{\brho\in S}\big\{\brho\cdot f(\bQ(t)\big\}-\bsigma(t)\cdot f(\bQ(t))\right).$$
This inequality shows that if $\bsigma(t)\cdot f(\bQ(t))$ is close to $\max_{\brho\in S}\big\{\brho\cdot f(\bQ(t)\big\}$ for most of time, $\Delta(\bx)$ is negative, i.e., $
L$ has the desired negative drift property. 

Next, we aim for bounding $\E_\bx[\Delta(\bx)]$. To this end, we consider the following event
\begin{equation*}
  \mathcal{E}_1 := \bigg\{ A_{\max}(0)+\dots+A_{\max}(\tau(\bx)-1)\leq (n+\mu\sqrt{n}+1)\,\tau(\bx) \bigg\},
\end{equation*}
where $A_{\max}(t):=\max\{1,A_1(t),\dots,A_n(t)\}$. 
The following lemma establishes the conditional expectation of $\Delta(\bx)$ given $\mathcal E_1$, which will be used later for bounding $\E_\bx[\Delta(\bx)]$.
Here, to simplify notation, we use $\pr_\bx[\mathcal{A}]$ to denote the conditional probability of event $\mathcal{A}$ under  the initial state $\{\bX(0)=\bx\}$. 
\begin{lemma} \label{lemma:bounds_on_E1}
For any $\alpha,\beta\in (0,1)$ and initial state $\bx=(\ba,\bW,\bQ)\in\Omega$ with large enough $Q_{\max}$, it follows that
  \begin{eqnarray}
  && \E_\bx[\Delta(\bx)\,|\,\mathcal{E}_1] ~\leq~ \left( -\frac{\varepsilon}{2}(1-\alpha)(1-\beta)+\frac{2n}{1-c}\big((1-\beta)\alpha+\beta\big) \right)f((1-c)Q_{\max})\tau(\bx), \label{eq:bound_delta_E1}\\
  && \pr_\bx[\mathcal{E}_1^c]\;\E_\bx[\Delta(\bx)\,|\,\mathcal{E}_1^c]
  ~\leq~ \left( \frac{n\,f(Q_{\max})}{\tau(\bx)}+n(n+1)+n\mu\sqrt{n}\right)\tau(\bx)\label{eq:bound_delta_E1c}.
\end{eqnarray}
\end{lemma}
\noindent The proof of the above lemma is given in Section \ref{sec:cond_exp_E}. 
A high level intuition for event $\mathcal{E}_1$ and above lemma is as follows. 
$A_{\max}(t)$ is at least the maximum change of queue length for each queue during $[t,t+1)$; that is, $|Q_{i}(t+1)-Q_{i}(t)|\leq A_{\max}(t)$, for every $i\in \mathcal{I}$. In other words, on $\mathcal{E}_1$, $Q_i(t)$ in $[0,\tau(\bx)]$ does not change too much. 
Namely, $\bW(t)\approx f(\bQ(t))$ does not change many times in $[0,\tau(\bx)]$ for $\bx=(\ba,\bW,\bQ)$ with large enough $Q_{\max}$. From condition {\bf C0} of the oracle system, the schedule $\bsigma(t)$ is close to a max-weight one with respect to $f(\bQ(t))$ `mostly' in the time interval $[0,\tau(\bx))]$, which guarantees the negative drift of $\Delta(\bx)$ on $\mathcal{E}_1$, i.e., \eqref{eq:bound_delta_E1}.
On the other hand, \eqref{eq:bound_delta_E1c} holds essentially because the event $\mathcal E_1$ occurs with high probability.
Now, we are ready to complete the proof of Lemma~\ref{lemma:negative_drift} using upper bounds \eqref{eq:bound_delta_E1} and \eqref{eq:bound_delta_E1c}.
For any $\bx=(\ba,\bW,\bQ)\in\Omega$ with large enough $Q_{\max}$, from \eqref{eq:L_Delta}, we have
\begin{eqnarray*}
\E_\bx[L(\tau(\bx))-L(0)]
	&\leq& \pr_\bx[\mathcal{E}_1]\ \E_\bx[ \Delta(\bx)\,|\,\mathcal{E}_1]
		+\  \pr_\bx[\mathcal E_1^c]\ \E_\bx[ \Delta(\bx)\,|\,\mathcal E_1^c]  + n\left(\mu^2+2\right)f'(0)\tau(\bx) \\[5pt]
	&\leq& \left( -\frac{\varepsilon}{2}(1-\alpha)(1-\beta)+\frac{2n}{1-c}\big((1-\beta)\alpha+\beta\big) \right)f((1-c)Q_{\max})\tau(\bx) \\
		 &&\;+\;\left( \frac{n\,f(Q_{\max})}{\tau(\bx)}+n(n+1)+n\mu\sqrt{n}\right)\tau(\bx)\\
		&&\mbox{} +\; n\left(\mu^2+2\right)f'(0)\tau(\bx)\\
	&\leq& \left( -\frac{\varepsilon}{2}(1-\alpha)(1-\beta)+\frac{2n}{1-c}\big((1-\beta)\alpha+\beta\big) \right)f((1-c)Q_{\max})\tau(\bx)\\
	&&\mbox{}+\; n\left( \frac{f(Q_{\max})}{\tau(\bx)} +(\mu^2+2)f'(0)+n+\mu\sqrt{n}+1 \right)\tau(\bx)\\
	  &=& -\kappa(\bx),
\end{eqnarray*}
which completes the proof of Lemma~\ref{lemma:negative_drift}.

\subsection{Proof of Lemma~\ref{lemma:bounds_on_E1}}\label{sec:cond_exp_E}
This subsection presents the proof of Lemma~\ref{lemma:bounds_on_E1}, thus completing the proof of Lemma~\ref{lemma:negative_drift}. 
In the proof of Lemma~\ref{lemma:bounds_on_E1}, we need two auxiliary results: Propositions~\ref{prop:stopping_time} and \ref{prop:main_lemma}. We prove theses propositions in Section~\ref{sec:pf_st} and \ref{sec:pr_ber}.

The following proposition states that $Q_{\max}(t)$ is bounded, and $\bW(t)$ changes at most $n$ times on $\mathcal{E}_1$.
\begin{proposition}\label{prop:stopping_time}
For any initial state $\bx=(\ba,\bW,\bQ)\in \Omega$ with large enough $Q_{\max}$, given that event $\mathcal E_1$ occurs, $\bW(t)$ changes at most $n$ times during $[0,\tau(\bx)]$ and 
\begin{equation}
(1-c) ~\leq~ \frac{Q_{\max}(t)}{Q_{\max}} ~\leq~ (1+c), \qquad\mbox{for all}~t\in [0,\tau(\bx)],\label{eq:bound_of_Q_max}
\end{equation}
where $c$ is the constant in condition {\bf C6} of Theorem~\ref{thm:mainResult}. 
\end{proposition}
\noindent
The proof of the above proposition is given in Section~\ref{sec:pf_st}. Let $T_m$ be the time at which the $m$-th change of weight vector $\bW(t)$ occurs, i.e.,
$\bW(t)$ remains fixed during the time interval $[T_m, T_{m+1})$. Formally, let $T_0=0$ and for $m\geq 1$, iteratively define
\begin{equation*}
  T_m\,:=\,\inf\{t\in\mathbb Z_+\,:\, \bW(t-1)\neq \bW(t),\; t>T_{m-1}\}.
\end{equation*}
Since
$\bW(t)$ remains fixed for $t\in[T_{m},T_{m+1})$, condition {\bf C0} implies that
that with high probability,
$\bsigma(t)$ is close to the max-weight schedule with respect to $\bW(t)$ for 
$T_m+h\leq t\leq T_{m+1}.$ Using this observation with Proposition~\ref{prop:stopping_time}, we obtain the following proposition, which states that with high probability, 
schedule $\bsigma(t)$ is close to a max-weight schedule with respect to $f(\bQ(t))$ `mostly' in time interval $[0,\tau(\bx))]$, on the event $\mathcal{E}_1$.
\begin{proposition}\label{prop:main_lemma}
For any $\eta,\alpha,\beta \in(0,1)$ and initial state $\bx=(\ba,\bW,\bQ)\in\Omega$ with
large enough $Q_{\max}$, it follows that
\begin{equation*}
  \pr_\bx\big[\left|T(\bx,\eta)\right|\geq\left(1-\alpha\right)\,\tau(\bx)\;|\;\mathcal{E}_1\big]~\geq~1-\beta,
\end{equation*}
where
\begin{equation}\label{eq:definition_T}
    T(\bx, \eta):=\bigg\{ t \in [0,\tau(\bx)]\,:\,
    \bsigma(t)\cdot f(\bQ(t))\geq (1-\eta)\,\max_{\brho\in S}\big\{\brho\cdot f(\bQ(t))\big\}\bigg\}.
  \end{equation}
\end{proposition}
\noindent
The proof of the above proposition is given in Section~\ref{sec:pr_ber}. In the remainder of this section, we derive \eqref{eq:bound_delta_E1} and \eqref{eq:bound_delta_E1c}
utilizing Propositions~\ref{prop:stopping_time} and \ref{prop:main_lemma}. 

First, from \eqref{eq:bound_of_Q_max}, we have the following upper bound for any summand in $\Delta(\bx)$: 
for all $t\in [0,\tau(\bx)]$,
\begin{eqnarray} 
   (1-\varepsilon)\,\max_{\brho\in S}\big\{\brho\cdot f(\bQ(t)\big\}-\bsigma(t)\cdot f(\bQ(t)) 
   &\leq& \max_{\brho\in S}\big\{\brho\cdot f(\bQ(t)\big\} \nonumber \\
   &\leq& n\,f(Q_{\max}(t)) \nonumber \\[4pt]
   &\leq& n\,f((1+c)Q_{\max}).\label{eq:general_bound}
\end{eqnarray}
Furthermore, we have a tighter bound for $t\in T(\bx,\eta)$. 
From the definition \eqref{eq:definition_T} of $T(\bx,\eta)$ with $\eta=\varepsilon/2$, we obtain that
for all $t\in T(\bx,\eta)$,
\begin{eqnarray}
  (1-\varepsilon)\,\max_{\brho\in S}\big\{\brho\cdot f(\bQ(t)\big\}-\bsigma(t)\cdot f(\bQ(t)) 
   &\leq& -\frac{\varepsilon}{2}\,\max_{\brho\in S}\big\{\brho\cdot f(\bQ(t))\big\} \nonumber\\
   &\leq& -\frac{\varepsilon}{2}\,f(Q_{\max}(t)) \nonumber \\
   &\leq& -\frac{\varepsilon}{2}\,f((1-c)Q_{\max}),\label{eq:special_bound}
\end{eqnarray}
where the last inequality comes from \,\eqref{eq:bound_of_Q_max}. 
Now, under $\eta=\varepsilon/2$ in Proposition~\ref{prop:main_lemma}, define the following event
\begin{eqnarray*}
  \mathcal{E}_2&:=&\left\{\,\left| \Big\{ t \in [0,\tau(\bx)]\,:\,
    \bsigma(t)\cdot f(\bQ(t))\geq (1-\varepsilon/2)\,\max_{\brho\in S}\big\{\brho\cdot f(\bQ(t))\big\}\Big\}  \right|\geq (1-\alpha)\tau(\bx)\,\right\}.    
\end{eqnarray*}
Then, according to Proposition~\ref{prop:main_lemma}, we have $\pr_\bx[\mathcal{E}_2\,|\,\mathcal{E}_1]\geq 1-\beta$ for $\bx=(\ba,\bW,\bQ)\in\Omega$ with large enough $Q_{\max}$. From upper bounds \eqref{eq:general_bound} and \eqref{eq:special_bound}, and the definition of the event $\mathcal{E}_2$, we conclude that 
\begin{eqnarray}
  \E_\bx[\Delta(\bx)\,|\,\mathcal{E}_1\cap \mathcal{E}_2] 
  &\leq& \E_\bx\!\!\left[ -\frac{\varepsilon}{2}\,f((1-c)Q_{\max})T(\bx,\eta)+  n\,f((1+c)Q_{\max})(\tau(\bx)-T(\bx,\eta))\,\bigg|\, \mathcal{E}_1\cap\mathcal{E}_2 \right] \nonumber \\
  &\leq& \left(-\frac{\varepsilon}{2}(1-\alpha)f((1-c)Q_{\max}) +\alpha\, n\,f((1+c)Q_{\max})\right)\tau(\bx).\label{eq:bound_E2}
\end{eqnarray}
From \eqref{eq:general_bound}, we also have
\begin{equation}
  \E_\bx[\Delta(\bx)\,|\,\mathcal{E}_2^c\cap\mathcal{E}_1]
  ~\leq~ \E_\bx\!\!\left[ \sum_{t=0}^{\tau(\bx)-1} n\,f((1+c)Q_{\max})\,\bigg|\,\mathcal{E}_2^c\cap\mathcal{E}_1 \right]
  ~=~\left( \,n\,f((1+c)Q_{\max}) \right)\tau(\bx).\label{eq:bound_E2c}
\end{equation}
\noindent
Using \eqref{eq:bound_E2} and \eqref{eq:bound_E2c}, we derive  \eqref{eq:bound_delta_E1} as follows:
\begin{eqnarray*}
\E_\bx[\Delta(\bx)\,|\,\mathcal{E}_1]
&=& 
\pr_\bx[\mathcal{E}_2\,|\,\mathcal{E}_1]\ \E_\bx[\Delta(\bx)\,|\,\mathcal{E}_1\cap \mathcal{E}_2] 
~+~ \pr_\bx[\mathcal{E}_2^c\,|\,\mathcal{E}_1]\ \E_\bx[\Delta(\bx)\,|\,\mathcal{E}_1\cap \mathcal{E}_2^c]\\[5pt]
  &\leq& 
    (1-\beta)\left(-\frac{\varepsilon}{2}(1-\alpha)f((1-c)Q_{\max}) +\alpha\, n\,f((1+c)Q_{\max})\right)\tau(\bx)\\[5pt]
    && \mbox{} +\;\beta \left( n\,f((1+c)Q_{\max}) \right)\tau(\bx) \\[5pt]
  &\leq& \left( -\frac{\varepsilon}{2}(1-\alpha)(1-\beta)+\frac{2n}{1-c}\big((1-\beta)\alpha+\beta\big) \right)f((1-c)Q_{\max})\tau(\bx),
\end{eqnarray*}
where, in the last inequality, we use the following property for concave functions $f$ with $f(0)=0$:
\begin{equation*}
  \frac{f((1+c)x)}{f((1-c)x)} \leq \frac{2}{1-c} \frac{f((1+c)x)}{f(2x)} \leq \frac{2}{1-c}.
\end{equation*}

Next, for proving 
\eqref{eq:bound_delta_E1c}, 
we introduce the following Chebyshev-type inequality involving conditional expectations in \cite{MR69}:
\begin{lemma}[\protect{\cite[Theorem 2.1]{MR69}}]\label{thm:Cheby}
If $X$ is a random variable with mean $\lambda$ and variance $\mu^2$ then we have
\begin{equation*}
    \left(\E[X|\mathcal{A}]-\lambda\right)^2~\leq~\mu^2\frac{1-p}{p},
\end{equation*}
for any event $\mathcal{A}$ with $\pr[\mathcal{A}]=p$.
\end{lemma}
\noindent
From conditions {\bf C1}, {\bf C2} and {\bf C4}, we have
\begin{eqnarray}
  \max_{\brho\in\mathcal S}\brho\cdot f(\bQ(t)) ~\leq~ n f(Q_{\max}(t)) 
  &\leq& n f(Q_{\max}+A_{\max}(1)+\dots+A_{\max}(\tau(\bx))) \nonumber \\
  &\leq& n f(Q_{\max})+nA_{\max}(1)+\dots+nA_{\max}(\tau(\bx)),\label{eq:outside_bound}
\end{eqnarray}
for $\bx\in\Omega$ with large enough $Q_{\max}$. 
Since $\E[A_{\max}(1)+\dots+A_{\max}(\tau(\bx))]\leq (n+1)\tau(\bx)$ and $\var[A_{\max}(1)+\dots+A_{\max}(\tau(\bx))]\leq n\,\mu^2\,\tau(\bx)$, we have
\begin{equation}
  \E\left[A_{\max}(1)+\dots+A_{\max}(\tau(\bx))\,|\, \mathcal{E}_1^c \right] ~\leq~ (n+1)\tau(\bx) + \mu\sqrt{n}\sqrt{\tau(\bx)}\sqrt{ \frac{1-\pr[\mathcal{E}_1^c]}{\pr[\mathcal{E}_1^c]}} \label{eq:chebyshev}
\end{equation}
from Lemma~\ref{thm:Cheby}. Also, according to the definition of event $\mathcal{E}_1$, we have
\begin{equation}
  \pr_\bx[\mathcal E_1^c]~\leq~ \frac{1}{\tau(\bx)}, \label{eq:event_E_1}
\end{equation}
the proof of which is elementary and given in Appendix~\ref{app:event_E_1} for completeness.
Combining 
\eqref{eq:outside_bound}, \eqref{eq:chebyshev}, and \eqref{eq:event_E_1}, we derive \eqref{eq:bound_delta_E1c} as follows:
\begin{eqnarray*}
  \pr_\bx[\mathcal{E}_1^c]\;\E_\bx[\Delta(\bx)\,|\,\mathcal{E}_1^c]
  &\leq& \pr_\bx[\mathcal{E}_1^c]\;\E_\bx\!\left[ \sum_{t=0}^{\tau(\bx)-1}\max_{\brho\in\mathcal S}\brho\cdot f(\bQ(t))\,\bigg|\,\mathcal E^c \right] \\
  &\leq& \pr_\bx[\mathcal{E}_1^c]\sum_{t=0}^{\tau(\bx)-1}\left( n\,f(Q_{\max})+n\,\E_\bx[A_{\max}(1)+\dots+A_{\max}(\tau(\bx))\,|\,\mathcal{E}^c ] \right)\\[5pt]
	&\leq& n\,f(Q_{\max}) +n\,\tau(\bx)\,\pr_\bx[\mathcal{E}_1^c]\,\E_\bx[A_{\max}(1)+\dots+A_{\max}(\tau(\bx))\,|\,\mathcal{E}^c ]  \\[5pt]
  &\leq& n\,f(Q_{\max}) +n\,\tau(\bx)\,\pr_\bx[\mathcal{E}_1^c]\left(  (n+1)\tau(\bx) + \mu\sqrt{n}\sqrt{\tau(\bx)}\sqrt{ \frac{1-\pr[\mathcal{E}_1^c]}{\pr[\mathcal{E}_1^c]}}\right)  \\[5pt]
  &\leq& n\,f(Q_{\max})+n(n+1)\,\tau(\bx) +n\mu\sqrt{n}\,\tau(\bx)\sqrt{\tau(\bx)}\sqrt{\pr[\mathcal{E}_1^c]}\\
  &\leq& n\,f(Q_{\max})+n(n+1)\,\tau(\bx) +n\mu\sqrt{n}\,\tau(\bx).
\end{eqnarray*}
This completes the proof of Lemma \ref{lemma:bounds_on_E1}.



\subsection{Proof of Proposition~\ref{prop:stopping_time}}\label{sec:pf_st}
This subsection presents the proof of Proposition~\ref{prop:stopping_time}. We assume that event $\mathcal E_1$ occurs throughout this section. 
We first note that, for all $i\in\mathcal I$ and $t\in[1,\tau(\bx)]$,
  \begin{align*}
  &Q_i(t)-Q_i(0)~\leq~A_{max}(0)+\dots+A_{max}(\tau(\bx)-1)~\leq~ (n+\mu\sqrt{n}+1)\tau(\bx),\\[0.05in]
  &Q_i(t)-Q_i(t-1)~\geq~-1,\\
  &\tau(\bx)~\leq~ \frac{c}{n+\mu\sqrt{n}+1}\,Q_{\max},
\end{align*}
where the right hand side of the last inequality is the second term of the minimum in the definition of $\tau(\bx)$. Then, we obtain
\begin{equation*}
  -c\,Q_{\max}~\leq~\tau(\bx)~\leq~ Q_i(t)-Q_i(0)~\leq~(n+\mu\sqrt{n}+1)\tau(\bx)~\leq~ c\,Q_{\max},\qquad\mbox{for all}~ t\in[0,\tau(\bx)],
\end{equation*}
which implies that \eqref{eq:bound_of_Q_max} in Proposition~\ref{prop:stopping_time} holds.

Now, we prove that for $\bx=(\ba,\bW,\bQ)$ with large enough $Q_{\max}$, $T_{n+1}>\tau(\bx)$, i.e., $\bW(t)$ changes at most $n$ times during $[0,\tau(\bx)]$. Toward this,
we claim that we need only to show that given initial state $\bX(0)=\bx=(\ba,\bW,\bQ)$ with large enough $Q_{\max}$, the following holds: 
\begin{equation}
  |U_i(t+1)-U_i(t)|~\leq~ f'(f^{-1}(g((1-c)\;Q_{\max}) ))\cdot A_{\max}(t),\qquad\mbox{for all}~ t\in[0,\tau(\bx)].\label{eq:bound_u2}
\end{equation}
Under assuming \eqref{eq:bound_u2}, one can obtain that for all $t\in[0,\tau(\bx)]$,
\begin{align*}
  |U_i(t)-U_i(0)| 
&~\leq~ f'(f^{-1}(g((1-c)\;Q_{\max}) )) \left( A_{\max}(1)+\dots+A_{\max}(\tau(\bx))\right)\\
&~\leq~ f'(f^{-1}(g((1-c)\;Q_{\max}) )) \left(\left(n+\mu\sqrt{n}+1\right) \tau(\bx)\right)~\leq~1,
\end{align*}
where the second inequality is from the definition of event $\mathcal{E}_1$ and the last inequality from 
the definition of $\tau(\bx)$. In other words, $U_i(t)$ varies by at most $1$ for all $i\in \mathcal{I}$ during $[0,\tau(\bx)]$. Then, since $W_i(t)$ is updated only if $U_i(t)$ varies by at least $2$, $W_i(t)$ changes at most once, and $\bW(t)$ changes at most $n$ times, which implies $T_{n+1}>\tau(\bx)$.
To verify \eqref{eq:bound_u2}, we investigate the variation of $U_i(t)$ by the following cases:
\begin{enumerate}
  \item[1.] Suppose that $f(Q_i(t))>g(Q_{\max}(t))$. Because $U_i(t)=f(Q_i(t))$ (from the definition of $U_i$), 
  $Q_i(t)>f^{-1}(g(Q_{\max}(t)))$ (from the previous assumption), $f'$ is decreasing (from condition {\bf C1}), and $Q_{\max}(t)\geq (1-c)Q_{\max}$ (from \eqref{eq:bound_of_Q_max}), we obtain an upper bound
  of $f'(Q_i(t))$ as:
  \begin{equation*}
    f^\prime(Q_i(t)) ~\leq~ f'(f^{-1}(g(Q_{\max}(t)))) ~\leq~ f'(f^{-1}(g((1-c)\;Q_{\max}) )).
  \end{equation*}
  
  \item[2.] Now, suppose that $f(Q_i(t))\leq g(Q_{\max}(t))$. Because $g'(x)<f'(x)$ for large enough $x$ (from condition {\bf C2}), $Q_{\max}(t)\geq (1-c)Q_{\max}$ (from \eqref{eq:bound_of_Q_max}), and $f'$ is decreasing (from condition {\bf C1}), we obtain an upper bound of $g'(Q_{\max}(t))$ as:
  \begin{equation*}
    g'(Q_{\max}(t)) ~\leq~ f'(Q_{\max}(t)) 
    ~\leq~  f'((1-c)\,Q_{\max}) 
    ~\leq~ f'(f^{-1}(g((1-c)\;Q_{\max}) ))
  \end{equation*}
  for large enough $Q_{\max}$. 
\end{enumerate}
Then, \eqref{eq:bound_u2} follows from the definitions of $U_i$ and $A_{\max}(t)$ and the above upper bounds of $f'(Q_i(t))$ and $g'(Q_{\max}(t))$.
This completes the proof of Proposition \ref{prop:stopping_time}.

\subsection{Proof of Proposition~\ref{prop:main_lemma}}\label{sec:pr_ber} 
This subsection presents the proof of Proposition~\ref{prop:main_lemma}.
Without loss of generality, assume that $\eta\leq \frac{7}{8}$ and let 
\begin{align*}
\eta'=1-(1-\eta)^{1/3}\leq \frac{1}{2},\quad
\alpha'=1-\sqrt{1-\alpha},\quad
\gamma=\frac{\beta}{n+1},\quad\textrm{and}\quad
\delta=\alpha'\gamma.
\end{align*} 
To simplify notation, we use $h(x)$ instead of $h(x,\eta',\delta)$. 
From conditions {\bf C1}-{\bf C6}, for large enough $x$, we have 
\begin{eqnarray}
  \frac{2n}{f\left((1-c)x\right)} \leq \eta', \label{eq:condition_1}\\
  \frac{n\,g((1+c)x)+2n}{\frac{1}{2}(f((1-c)x)-2)} \leq \eta', \label{eq:condition_2}\\
  \frac{(n+1)\;h(f((1+c)x+2))}{cx} \leq \alpha', \label{eq:condition_3}\\
  (n+1)\;f'(f^{-1}(g((1-c)x)))\;h(f((1+c)x+2)) \leq \alpha', \label{eq:condition_4}
\end{eqnarray}
where their detailed proofs are given in Appendix \ref{app:condition}.

In the rest of this section, we assume that $\mathcal{E}_1$ occurs.
Then, we claim that the following conditions are sufficient to prove Proposition~\ref{prop:main_lemma}: 
\begin{enumerate}
  \item[(a)] For all $t\in[0,\tau(\bx)]$,\begin{equation*}
    (1-\eta') \left[\max_{\brho\in S} \brho \cdot f(\bQ(t))\right] \leq \max_{\brho\in S} \brho \cdot \bW(t).
  \end{equation*}
  \item[(b)] With probability at least $1-\beta$, at least $(1-\alpha)\tau(\bx)$ number of time instance $t \in [0,\tau(\bx)]$ satisfy
  \begin{equation}
    (1-\eta')\left[\max_{\brho\in S}\brho\cdot\bW(t)\right]\leq \bsigma(t)\cdot\bW(t).\label{eq:claimb}
  \end{equation}
  \item[(c)] For all $t\in [0,\tau(\bx)]$ at which \eqref{eq:claimb} is satisfied,
  \begin{equation*}
    (1-\eta')\left(\bsigma(t)\cdot\bW(t)\right)\leq \bsigma(t)\cdot f(\bQ(t)).
  \end{equation*}
\end{enumerate}
The proof of Proposition~\ref{prop:main_lemma} comes immediately from (a), (b) and (c):
\begin{eqnarray*}
  (1-\eta)\left[\max_{\brho\in S}\brho\cdot f(\bQ(t))\right]
  &=&(1-\eta')^3\left[\max_{\brho\in S} \brho \cdot f(\bQ(t))\right] \\
  &\leq&  (1-\eta')^2\left[\max_{\brho\in S} \brho \cdot \bW(t) \right]\\
  &\leq& (1-\eta') \left(\bsigma(t)\cdot\bW(t)\right) \\
  &\leq& \bsigma(t)\cdot f(\bQ(t)),
\end{eqnarray*}
where 
with probability at least $1-\beta$, at least $(1-\alpha)\tau(\bx)$ number of time instance $t \in [0,\tau(\bx)]$ satisfy
the second last and last inequalities. 
Hence, we proceed toward proving (a), (b) and (c).

\smallskip
\noindent {\bf Proof of (a).}
Recall that our scheduling algorithm in Section \ref{sec:algorithm_description}
maintains $|U_i(t)-W_i(t)|\leq 2$, where $U_i(t)=\max\{f(Q_i(t)), g(Q_{\max}(t))\}$. 
Thus, for $t\in\mathbb Z_+$ and $i\in \mathcal{I}$, we have $f(Q_i(t))-2 ~\leq~ W_i(t)$, and hence
\begin{equation}
  \max_{\brho\in S} \brho\cdot f(\bQ(t))-2n ~\leq~ \max_{\brho\in S} \brho \cdot \bW(t).\label{eq:rel_f_W}
\end{equation}
In addition, from Proposition~\ref{prop:stopping_time}, we have
\begin{equation}
  f((1-c)Q_{\max})~\leq~f(Q_{\max}(t))~\leq~\max_{\brho\in S} \brho\cdot f(\bQ(t)).\label{eq:f_max_lower}
\end{equation}
Therefore, we conclude that, for large enough $Q_{\max}$, 

\begin{eqnarray*}
  (1-\eta') \left[\max_{\brho\in S} \brho \cdot f(\bQ(t))\right]
  &\leq& \left(1-\frac{2n}{f((1-c)Q_{\max})}\right) \left[\max_{\brho\in S} \brho \cdot f(\bQ(t))\right]\\
  &=& \max_{\brho\in S} \brho\cdot f(\bQ(t))-2n\left( \frac{\max_{\brho\in S} \brho \cdot f(\bQ(t))}{f((1-c)Q_{\max})} \right)\\
  &\leq& \max_{\brho\in S} \brho\cdot f(\bQ(t))-2n\\
  &\leq& \max_{\brho\in S} \brho \cdot \bW(t),
\end{eqnarray*}
where the first inequality comes from \eqref{eq:condition_1}, the second inequality from \eqref{eq:f_max_lower}, and the last inequality from \eqref{eq:rel_f_W}.
  


\smallskip
\noindent {\bf Proof of (b).}
Recall that $T_m$ is the time at which the $m$-th change of weight vector $\bW(t)$. For $t\in[T_m,T_{m+1})$, let $\bW:=\bW(t)$ and
define a binary random variable $Z_t\in\{0,1\}$ by
\begin{equation*}
  Z_t~:=~\begin{cases}
    1 & \quad\textrm{if $\left( \bsigma_{\oracle}(\ba_{\oracle}^{(t)}(\ba))\right)
      \cdot \bW
       ~<~(1-\eta') \max_{\brho\in S}
      \brho\cdot \bW
  $} \\
  0 & \quad\textrm{otherwise}
  \end{cases}.
\end{equation*}
Then, from condition {\bf C0}, for $t\in[T_m+h',T_{m+1})$, we have $\E[Z_t]< \delta$ and $\E\left[\sum_{t=h'}^{l-1} Z_t\right]< \delta(l-h')$, where $h'\geq h(W_{\max},\eta',\delta)$ and $l> h'$.
Applying the Markov inequality to the random variable $\sum_{t=h}^{l-1} Z_t$, we conclude that 
\begin{align*}
  \bigg|\big\{t\in [h,l)~:~ \left( \bsigma_{\oracle}(\ba_{\oracle}^{(t)}(\ba))\right)
      \cdot \bW
       \geq(1-\eta) \max_{\brho\in S}
      \brho\cdot \bW\big\}\bigg|~ \geq~ (1-\delta/\gamma) (l-h)
\end{align*}
occurs with probability at least $1-\gamma$.
In other words, with probability $\geq 1-\gamma$, 
\begin{align*}
  \bigg|\big\{t\in[T_i+h,T_{i+1}):\bsigma(t)\cdot\bW(t)
  \geq(1-\eta')\max_{\brho\in S}\brho\cdot\bW(t)\big\}\bigg|  \geq (1-\alpha')(T_{i+1}-T_i-h),
\end{align*}
where from Proposition~\ref{prop:stopping_time}, 
we set 
$$h~=~h((1+c)Q_{\max}+2,\eta',\delta)~\geq~h(\bW(t),\eta',\delta).$$ 
Since $\bW(t)$ changes at most $n$ times in $[0,\tau(\bx)]$ from Proposition~\ref{prop:stopping_time},
one can use the union bound and conclude that
with probability $\geq 1-\beta=1-(n+1)\gamma$, 
$$(1-\eta')\max\brho\cdot \bW(t)\leq \bsigma(t)\cdot\bW(t),$$ 
for at least $(1-\alpha')$ fraction of times in $\bigcup_{i=0}^{n}[T_i+h,T_{i+1})$. 
Furthermore, from \eqref{eq:condition_3}, \eqref{eq:condition_4}, and the definition of $\tau$, we have
\begin{eqnarray*}
  (n+1)h&=&{(n+1)h(f(1-c)Q_{\max}+2)} \\
  &\leq& \frac{\alpha'}{2}\min\left\{c\,Q_{\max},\frac{1}{f'(f^{-1}(g(1-c)Q_{\max}))}\right\}\\
  &\leq& \frac{\alpha'}{2}(\tau(\bx)+1)~\leq~\alpha'\tau(\bx).
\end{eqnarray*}
Thus, it follows that
\begin{eqnarray*}
  \left|\bigcup_{i=0}^{n}[T_i+h,T_{i+1})\right|&\geq&\tau(\bx)-(n+1)h
  \geq (1-\alpha')\tau(\bx).
\end{eqnarray*}
Therefore, with probability $\geq 1-\beta$, for at least $(1-\alpha')^2=(1-\alpha)$ fraction of times in the interval $[0,\tau(\bx)]$, 
\eqref{eq:claimb} holds.

\smallskip
\noindent {\bf Proof of (c).} 
From $|U_i(t)-W_i(t)|\leq 2$, where $U_i(t)=\max\{f(Q_i(t)), g(Q_{\max}(t))\}$,  we have
\begin{equation}
  f(Q_i(t))-2 ~\leq~ W_i(t) ~\leq~ f(Q_i(t))+g(Q_{\max}(t))+2, \label{eq:rel_f_W_both}
\end{equation}
Then, for all $t\in [0,\tau(\bx)]$ at which \eqref{eq:claimb} is satisfied, we obtain
\begin{eqnarray}
  \frac{1}{2}(f((1-c)Q_{\max})-2) &\leq& \frac{1}{2} (f(Q_{\max}(t))-2)~\leq~ \frac{1}{2} \left[ \max_{\brho\in S} \brho\cdot \bW(t)\right] \nonumber \\
  &\leq& \eta'\left[ \max_{\brho\in S} \brho\cdot \bW(t)\right]~\leq~ \bsigma(t)\cdot\bW(t). \label{eq:claimc}
\end{eqnarray}
where the first inequality comes from $(1-c)Q_{\max}\leq Q_{\max}(t)$ in Proposition~\ref{prop:stopping_time}
, the second inequality from \eqref{eq:rel_f_W_both}, the third inequality from the assumption $\eta'\leq 1/2$, and the last inequality from \eqref{eq:claimb}. We also have
\begin{eqnarray}   
  \bsigma(t)\cdot \bW(t) &\leq&  \bsigma(t)\cdot f(\bQ(t))+n\,g(Q_{\max}(t))+2n \nonumber \\
  &\leq& \bsigma(t)\cdot f(\bQ(t))+n\,g((1+c)Q_{\max})+2n, \label{eq:f_W_upper}
\end{eqnarray}
where the first inequality follows from \eqref{eq:rel_f_W_both} and the second inequality from $Q_{\max}(t)\leq (1+c)Q_{\max}$ in Proposition~\ref{prop:stopping_time}. The last two inequalities with \eqref{eq:condition_2} lead to (c): 
for large enough $Q_{\max}$, we have
\begin{eqnarray*}
  (1-\eta') \bsigma(t)\cdot\bW(t)
  &\leq& \left(1-\frac{n\,g((1+c)Q_{\max})+2n}{\frac{1}{2}(f((1-c)Q_{\max})-2)}\right)\bsigma(t)\cdot\bW(t) \\
  &\leq& \bsigma(t)\cdot\bW(t)-n\,g((1+c)Q_{\max})+2n\\
  &\leq& \bsigma(t)\cdot f(\bQ(t)),
\end{eqnarray*}
where the first inequality comes from \eqref{eq:condition_2}, the second inequality from \eqref{eq:claimc}, and the last inequality from \eqref{eq:f_W_upper}.
This completes the proof of Proposition~\ref{prop:main_lemma}.

\section{Conclusion}

The problem of dynamic resource allocation among network users contending resources 
has long been the subject of significant research in the last four decades. In this paper, we develop  a generic framework
for designing resource allocation algorithms of low-complexity and high-performance via
connecting iterative optimization methods and scheduling algorithms. 
Our work establishes sufficient conditions on queue-length functions
so that a queue-based scheduling algorithm is throughput-optimal.
To our best knowledge, our result is the first that establishes a rigorous connection between iterative optimization methods and 
low-complexity scheduling algorithms. We believe that it is of broader interest to design
low-complexity scheduling algorithms with high performance in various domains.

\bibliographystyle{plain}



\newpage
\appendix 
\section{Proof of \eqref{eq:tau_infinity} and \eqref{eq:kappa_infinity}}\label{app:infty}
This section verifies two equations
\begin{align*}
  \lim_{L(\bx)\to\infty} \tau(\bx)&~=~\infty  \tag{\ref{eq:tau_infinity} Revisited}\\
  \lim_{L(\bx)\to\infty} \kappa(\bx)/\tau(\bx)&~=~\infty, \tag{\ref{eq:kappa_infinity} Revisited}
\end{align*}
where, for $\bx=(\ba,\bW,\bQ)\in\Omega$, $L(\bx)=\sum_{i\in \mathcal{I}}\int_0^{Q_i} f(s)d s$,
\begin{align*}
   \tau(\bx)&~=~\left\lfloor \frac{1}{(n+s\sqrt{n})+1}\,\min\left\{ \frac{1}{f'\left(f^{-1}(g((1-c)Q_{\max}) ))\right)},~ c\, Q_{\max}  \right\} \right\rfloor,  \tag{\ref{eq:deftau} Revisited}\\
   \frac{\kappa(\bx)}{\tau(\bx)}
	&~=~ \left( \frac{\varepsilon}{2}(1-\alpha)(1-\beta)+\frac{2n}{1-c}\big((1-\beta)\alpha+\beta\big) \right)f((1-c)Q_{\max}) \\
	&~~\mbox{}~- n\left( \frac{f(Q_{\max})}{\tau(\bx)} +(\mu^2+2)f'(0)+n+\mu\sqrt{n}+1 \right), \tag{\ref{eq:defkappa} Revisited}
\end{align*}
and $\alpha$, $\beta$ are constants that satisfy
\begin{equation}
  \frac{\varepsilon}{2}(1-\beta)(1-\alpha)-\frac{2n(\beta+(1-\beta)\alpha)}{1-c}>0. \label{eq:alpha_beta}
\end{equation}
We first note that $L(\bx)\to\infty$ if and only if $Q_{\max}\to\infty$ from the definition of $L$. To show that \eqref{eq:tau_infinity}, we calculate the limit of $\tau(\bx)$ in cases:
\begin{enumerate}
	\item[(i)] If $c\,Q_{\max}\leq \frac{1}{f'\left(f^{-1}(g((1-c)Q_{\max}) ))\right)}$, we have
	\begin{equation*}
		\lim_{L(\bx)\to\infty} \left\lfloor \frac{1}{(n+s\sqrt{n})+1}c\,Q_{\max} \right\rfloor
		~\geq~\lim_{Q_{\max}\to\infty} \frac{1}{(n+s\sqrt{n})+1}c\,Q_{\max}-1~=~\infty,
	\end{equation*}
	\item[(ii)] If $c\,Q_{\max}>\frac{1}{f'\left(f^{-1}(g((1-c)Q_{\max}) ))\right)}$, we have
	\begin{eqnarray*}
		\lefteqn{\lim_{L(\bx)\to\infty} \left\lfloor \frac{1}{(n+s\sqrt{n})+1}\frac{1}{f'\left(f^{-1}(g((1-c)Q_{\max}) ))\right)}\right\rfloor}\\
		&\geq& \lim_{Q_{\max}\to\infty} \frac{1}{(n+s\sqrt{n})+1}\frac{1}{f'\left(f^{-1}(g((1-c)Q_{\max}) ))\right)}-1~=~\infty
	\end{eqnarray*}
	since  $\lim_{x\to\infty} f(x)=g(x)=\infty$ and $\lim_{x\to\infty} f'(x)=0$ (conditions {\bf C1}, {\bf C2}, and {\bf C4}).
\end{enumerate}
Therefore, we have $\lim_{L(\bx)\to\infty} \tau(\bx)=\infty$.

\smallskip
To prove \eqref{eq:kappa_infinity}, note that the following property for concave function $f$ with $f(0)=0$:
\begin{equation*}
	f((1-c)x)~\geq~(1-c)f(x)+cf(0)~=~(1-c)f(x).
\end{equation*}
Then, we have
\begin{eqnarray*}
	\lim_{L(\bx)\to\infty} \frac{\kappa(\bx)}{\tau(\bx)}
	&=&  \lim_{Q_{\max}\to\infty} \left( \frac{\varepsilon}{2}(1-\alpha)(1-\beta)+\frac{2n}{1-c}\big((1-\beta)\alpha+\beta\big) \right)f((1-c)Q_{\max}) \\
	&&\mbox{}- \frac{n}{\tau(\bx)}f(Q_{\max})-n\left( (\mu^2+2)f'(0)+n+\mu\sqrt{n}+1 \right)\\
	&\geq& \lim_{Q_{\max}\to\infty} \left(\frac{\varepsilon}{2}(1-\alpha)(1-\beta)+\frac{2n}{1-c}\big((1-\beta)\alpha+\beta\big)-\frac{1}{1-c}\frac{n}{\tau(\bx)}\right)f((1-c)Q_{\max})\\
	&&\mbox{}-n\left( (\mu^2+2)f'(0)+n+\mu\sqrt{n}+1 \right)~=~\infty,
\end{eqnarray*}
due to \eqref{eq:alpha_beta}, \eqref{eq:tau_infinity}, and $\lim_{x\to\infty} f(x)=\infty$.

\section{Proof of \eqref{eq:event_E_1}}\label{app:event_E_1}
This section proves \eqref{eq:event_E_1}; that is, $\pr_\bx[\mathcal E_1^c]\leq \frac{1}{\tau(\bx)}$, where $\mathcal{E}_1=\{ A_{\max}(0)+\dots+A_{\max}(\tau(\bx)-1)\leq (n+\mu\sqrt{n}+1)\,\tau(\bx)\}$.
Recall that $A_{\max}(t)=\max\{1,A_1(t),\dots,A_n(t) \}$. Then, we have $\E[A_{\max}(t)]\leq \sum_{i=1}^n \lambda_i+1\leq n+1$ and $\var\left[A_{\max}(t)\right]\leq \sum_{i=1}^n \var[A_i(t)]\leq n\mu^2$ for every $t\in\mathbb Z_+$. 
From these inequalities and Chebyshev's inequality, we have 
\begin{eqnarray*}
  \pr_\bx\left[ \mathcal E_1^c \right]
  &=&\pr_\bx\left[A_{\max}(1)+\dots+A_{\max}(\tau(\bx))  \geq \left(n+\mu\sqrt{n}+1\right)\tau(\bx)\right]\\
  &=&\pr_\bx\left[A_{\max}(1)+\dots+A_{\max}(\tau(\bx))  \geq (n+1)\,\tau(\bx)+\sqrt{\tau(\bx)}\left(\mu\sqrt{n}\sqrt{\tau(\bx)}\right)\right]
  ~\leq~ \frac{1}{\tau(\bx)},
\end{eqnarray*}
which verifies \eqref{eq:event_E_1}.

\section{Proof of \eqref{eq:condition_1}--\eqref{eq:condition_4}}\label{app:condition}

Conditions {\bf C2} and {\bf C5} show that  
\begin{eqnarray*}
  \lim_{x\to\infty}  \frac{2n}{f\left((1-c)x\right)} =0,\\
  \lim_{x\to\infty}   \frac{h(f((1+c)x+2))}{x} =0,
\end{eqnarray*}
which implies \eqref{eq:condition_1} and \eqref{eq:condition_3}. Now, from conditions {\bf C1} and {\bf C2}, we obtain
\begin{equation*}
  \lim_{x\to\infty} \frac{g((1+c)x)}{f((1-c)x)} \leq \lim_{x\to\infty} \frac{2}{1-c}\frac{g((1+c)x)}{f(1+c)x}= 0,
\end{equation*}
where we use the following property for concave function $f$:
\begin{equation*}
  \frac{f((1+c)x)}{f((1-c)x)} \leq \frac{1}{1-c} \frac{f((1+c)x)}{f(2x)} \leq \frac{2}{1-c}.
\end{equation*}
Then, \eqref{eq:condition_2} holds.
Finally, for \eqref{eq:condition_4}, we let $x'=x+2/(1+c)>x$, and we observe that
\begin{eqnarray*}
  \lefteqn{f'(f^{-1}(g((1-c)x)))\;h(f((1+c)x+2))}\\
  &=& f'(f^{-1}(g((1-c)x)))\;h(f((1+c)x')) \\
  &=& \frac{f'(f^{-1}(g((1-c)x)))}{f'(f^{-1}(g((1-c)x')))} \\
  &&\mbox{}  \times f'(f^{-1}(g((1-c)x')))\;h(f((1+c)x')) \\
  &\leq& f'(f^{-1}(g((1-c)x')))\;h(f((1+c)x'))\to 0,
\end{eqnarray*}
as $x\to\infty$ from condition {\bf C6}. This completes the proof of \eqref{eq:condition_4}.

\end{document}